\documentclass[a4paper,10pt]{article}
\usepackage{fullpage}

\usepackage[utf8]{inputenc}
\usepackage{amsfonts}
\usepackage{amssymb}
\usepackage{amsmath}
\usepackage{amsthm}
\usepackage{tikz}
\usepackage{url}
\usepackage{setspace}
\usepackage{enumerate}

\usepackage{calligra}
\usepackage[T1]{fontenc}
\usepackage[all]{xy}

\theoremstyle{plain}
\newtheorem{defi}{Definition}[section]
\newtheorem{thm}{Theorem}[section]
\newtheorem{prop}{Proposition}[section]

\newtheorem{coro}{Corollary}[section]
\newtheorem{rem}{Remark}[section]
\newcommand{\X}{\mathbb{X}}
\newcommand{\N}{\mathbb{N}}
\newcommand{\Z}{\mathbb{Z}}

\newcommand{\R}{\mathbb{R}}
\newcommand{\Q}{\mathbb{Q}}
\newcommand{\F}{\mathbb{F}}

\DeclareMathOperator{\Gl}{Gl}
\DeclareMathOperator{\Pic}{Pic}
\DeclareMathOperator{\Jac}{Jac}

\DeclareMathOperator{\Proj}{Proj}
\DeclareMathOperator{\Spec}{Spec}
\DeclareMathOperator{\Ker}{Ker}

\DeclareMathOperator{\GCD}{gcd}
\DeclareMathOperator{\card}{card}

\begin{document}
\title{Good rings and homogeneous polynomials }
\author{J. Fresnel and M. Matignon}

\maketitle

\begin{abstract} 
In 2011, Khurana, Lam and Wang define the following property.
(*) {\sl A  commutative unital ring $A$ satisfies the property} ``power stable 
range one'' {\sl if for all $a,b\in A$ with $aA+bA=A$ there is an integer 
$N=N(a,b)\geq 1$ and $\lambda=\lambda (a,b) \in A$ such that $b^N+\lambda a\in 
A^\times$, 
the unit group of $A$.} 

 In 2019, Berman and Erman consider rings with the following property 
\newline 
(**) {\sl A commutative unital ring $A$ has enough homogeneous polynomials if 
for any $k\geq 1$ and set $S:=\{p_1,p_2,...,p_k\}$, of 
primitive points in $A^n$ and any $n\geq 2$, there exists an homogeneous polynomial 
$P(X_1,X_2,...,X_n)\in 
A[X_1,X_2,...,X_n]$ with $\deg P\geq 1$  and $P(p_i)\in A^\times$ for $1\leq 
i\leq k$.} 

We show in this article that the two properties (*) and (**) are equivalent and 
we shall call a 
commutative unital ring with these properties a good ring. 

When $A$ is a commutative unital ring  of pictorsion as defined by 
Gabber, Lorenzini and Liu in 2015, we show that $A$ is a good ring. 
  Using a Dedekind domain built by Goldman in 1963, we show that the 
converse is false.

\end{abstract} 

\maketitle

\section{Introduction}

In this paper we consider only commutative and unital  rings. As usual 
$A^\times$ denotes the unit group of $A$ and ring homomorphisms send $1$ to 
$1$. In particular if $A$ is the ring reduced to $\{0\}$ then $A^\times=A$.
In general we follow the notations in \cite{L}. 

The main goal of this paper is to study a class of rings that we will call ``good rings'' and to analyse their relations with some classical or less classical properties. 
\medskip

\noindent
{\bf A. Property ``power stable range one''. Good points and good rings.} In 2011, Khurana, Lam and Wang (\cite{KLW}, Definition 1.1 p. 123) were 
interested in the notion of 
``rings of square stable range one'' which can be seen as an extension of the 
notion ``$n$ is the stable range of a ring'' as defined by Bass in 1964 
(\cite{B}, p. 498).

One says that a ring $A$ satisfies the property ``square stable range one'' if 
for all $a,b\in A$ with $aA+bA=A$, there is $\lambda \in A$ such that 
$b^2+\lambda a\in A^\times$, where $A^\times$.

In the epilogue of their paper (\cite{KLW}, p. 141) they give a generalization 
of the 
property 
``square stable range one'',
namely

\begin{defi}\label{def.PSR}{\bf Property ``power stable range one''.} A  ring $A$ satisfies the property {\rm power stable range one},
if for all $a,b\in A$ with $aA+bA=A$,
there is an integer $N=N(a,b)\geq 1,\ \lambda =\lambda (a,b) \in A$ with 
$b^N+\lambda a\in 
A^\times$, the unit group of $A$.
\end{defi}

Let us re-interpret this notion in terms of primitive points.

\begin{defi}\label{def.Good}{\bf  Good points and good rings.} Let $A$ be a ring. Recall that a point \newline $p=(x_1,x_2,...,x_n)\in A^n$ 
is 
 {\sl primitive} if $\sum_{1\leq j\leq n} x_{j}A=A$. 
\begin{enumerate}[1)]
\item A primitive point 
$(a,b) \in A^2$ is a {\sl good 
point}  if there is an integer
$N=N(a,b)\geq 1$ and \newline $\lambda=\lambda (a,b) \in A$ with $b^N+\lambda a\in 
A^\times$.
\item 
A ring $A$ is a {\sl good ring} 
if  the primitive points in $A^2$ are good points.
\end{enumerate}
\end{defi}

So, in other words, good rings are those satisfying the power stable range one property. 
\begin{rem}\label{A} The ring $\Z$ of integers is a good ring. Namely, if $(a,b) \in \Z^2$ is a primitive 
point, $b$ is a unit modulo $a\Z$ and as the unit group $\frac{\Z}{a\Z}^\times$ is finite, then $(a,b)$ is a good point.
\end{rem}
\medskip
\noindent
{\bf B. Rings with enough homogeneous polynomials.}
Before going further, we remark that if $A$ is a  ring and 
$p:=(x_1,x_2,...,x_n)\in A^n$
is such that $P(p):=P(x_1,x_2,...,x_n)\in  A^\times$ for some homogeneous 
polynomial
$P(X_1,X_2,...,X_n)\in 
A[X_1,X_2,...,X_n]$ of $\deg P\geq 1$, then $p$ is a primitive point. Reciprocally if 
$p:=(x_1,x_2,...,x_n)\in A^n$ is primitive, then there is 
$W(X_1,X_2,...,X_n):=\sum_{1\leq i\leq n}u_iX_i$ such that \newline $W(p)\in A^\times$. 

Now generalizing this equivalence, we define a new family of rings as suggested 
in \cite{BE}, namely

\begin{defi}\label{def.Enough}{\bf Rings with enough homogeneous polynomials.}
 A ring $A$ has {\rm enough homogeneous polynomials in two 
variables}, (resp. {\rm enough homogeneous polynomials}) if for all  finite set 
 $S:=\{p_1,p_2,...,p_k\}$ of  primitive points in $A^2$ with  
 $\card S:=k\geq 1$, 
(resp.  primitive points in $A^n$ and any $n\geq 2$), there is $P(X_1,X_2)\in 
A[X_1,X_2]$ (resp. $P(X_1,X_2,...,X_n)\in A[X_1,X_2,...,X_n]$) with $P$ 
homogeneous, $\deg P\geq 1$  and $P(p_i)\in A^\times$ for $1\leq i\leq k$ where 
$P(p_i):=P(p_{1,i},p_{2,i},...,p_{n,i})$ and 
$p_i:=(p_{1,i},p_{2,i},...,p_{n,i})\in A^n$.
\end{defi}

A main result is 
\vskip 1mm \vskip 0pt
\noindent {\bf Theorem \ref{thm2.1}.}
{\em 
Let $A$ be a ring. The following 
properties are 
equivalent.
\begin{enumerate}[i)]
\item The ring $A$ is a good ring,
\item the ring $A$ has enough homogeneous polynomials.
\item the ring $A$ has enough homogeneous polynomials in two variables.
\end{enumerate}
}

\begin{rem}\label{B} In the case $A=\Z$, we saw (Remark \ref{A}) that $\Z$ is a good ring. The implication i) implies ii) in Theorem \ref{thm2.1} works by induction on the cardinality of the set of points we want to interpolate contrarily to the proof the ring $\Z$ has enough homogeneous 
polynomials
in \cite{BE}, Theorem 0.1. Let us sketch the steps in their proof. The first point is to show that a field $K$ has enough homogeneous 
polynomials. This follows from the classical avoiding lemma namely that if an ideal of a ring is included 
in a finite union of prime ideals then it is included in one of them. Then they show that the property 
``having enough polynomials'' transfers to the direct product of two rings and deduce that for $a\in \Z-\{0,\pm 1\}$
the ring $\frac{\Z}{a\Z}$ has enough homogeneous polynomials. Now let $S:=\{p_1,p_2,...,p_r\}$ primitive points in 
$\Z^n$. Again the avoiding lemma proves the existence of $P_i\in \Z[X_0,X_1,...,X_n]$ homogeneous and non constant
such that $P_i(p_i)\neq 0$ for $1\leq i\leq r$ and $P_i(p_j)=0$ for $j\neq i$. Let $a:=P_1(p_1)...P_r(p_r)\neq 0 $ and $\overline p_i $ the image of 
$p_i$ in as $\frac{\Z}{a\Z}^n$. As $\frac{\Z}{a\Z}$ has enough homogeneous polynomials there is 
$H\in \frac{\Z}{a\Z}[X_0,X_1,...,X_n]$,  homogeneous and non constant
such that $H(\overline p_i)\in \frac{\Z}{a\Z}^\times$. Then combining the $P_i$ with a lifting of $H$ in 
$\Z[X_0,X_1,...,X_n],$ they get a non constant homogeneous polynomial $Q$ with $Q(p_i)=1$ for all $i$. 
\end{rem}
\noindent
{\bf C. Pictorsion rings.} The last class is related with torsion in Picard groups, namely with (\cite{GLL}, Definition 0.3, p. 1191) we can define 

\begin{defi}\label{def.Pictorsion}{\bf Pictorsion rings.}
A ring $A$ is {\sl a pictorsion ring} if for all ring $B$ which is 
finite over $A$, its Picard group $\Pic (B)$ is 
a torsion group.
\end{defi}

\begin{rem}\label{C} We would like to comment on this notion. Namely, why is it usefull?  Is the ring 
$\Z$ a pictorsion ring?

This notion is usefull in order to have a Noether normalization lemma for families ((\cite{CMPT}, 
\cite{GLL}).

If $R$ is a pictorsion ring then Noether normalization lemma is valid for  projective 
schemes over $\Spec R$. Moreover a ring $R$  is pictorsion if for all equidimensional  projective 
schemes over $\Spec R$ there is a Noether normalization ((\cite{CMPT}, 
\cite{GLL}). 

Now what's about $\Z$ and pictorsion ?

A long time ago, Minkowski using the geometry of numbers proves the finiteness of the class group of 
the integral closure of $\Z$ in a finite  algebraic extension of $\Q$ and more generally the finiteness of the picard
group of a ring which is finite over $\Z$ appears in (\cite{MB}, Theorem 2.3. p. 165) in his geometric proof 
of Rumely's theorem on Skolem problem. So we can say 
that $\Z$ is a ring of 
pictorsion. 

Note there is  a probalistic approach in  \cite{BrE} to Noether normalization lemma for the rings 
$\Z$ or $\F_q[T]$. 
\end{rem}

It is more subtle to express the property of good point with pictorsion, namely
we prove the following caracterization of good points. 
\vskip 1mm \vskip 0pt
\noindent {\bf Theorem \ref{thm3.1}.}
{\em 
Let $A$ be a ring, $(a,b)\in A^2$ a primitive point. 

Let $A[x,y]:=\frac{A[X,Y]}{X(aY-bX)A[X,Y]}$ where $x$ (resp.$y$) is the image 
of $X$ (resp. $Y$) by the natural epimorphism. Moreover $A[x,y]$ is endowed 
with the induced grading.
Let $S(a,b):=\Proj A[x,y]$. The following properties are equivalent.
\begin{enumerate}[i)]
\item  The ${\cal O}_{S(a,b)}(S(a,b))$-module ${\cal O}_{S(a,b)}(1)(S(a,b))$ is 
a torsion element in the Picard group of \newline ${\cal O}_{S(a,b)}(S(a,b)),$
\item  there exists $P(X,Y)\in A[X,Y]$ homogeneous of degree $d\geq 1$ with 
$P(0,1),\ 
P(a,b)\in A^{\times},$
\item   the point $(a,b)\in A^2$ is a good point.
\end{enumerate}
}
 
It follows that the ring $A$ 
is 
a good ring 
if and only if for all primitive point $(a,b)\in A^2$, \newline ${\cal O}_{S(a,b)}(1)(S(a,b))$ is a torsion element in the 
Picard group of ${\cal O}_{S(a,b)}(S(a,b))$.

More specifically, let $A$ be a pictorsion ring.  As ${\cal 
O}_{S(a,b)}(S(a,b))$ is a free rank two 
$A$-module (Proposition \ref{prop3.1}), it follows that ${\cal 
O}_{S(a,b)}(1)(S(a,b))$ is a 
torsion element in $\Pic ({\cal O}_{S(a,b)}(S(a,b)))$. We get
\vskip 1mm \vskip 0pt
\noindent {\bf Corollary \ref{coro3.1}.}
{\em 
Let $A$ be a pictorsion ring, then  $A$ is a good ring.
}
\vskip 1mm \vskip 0pt
A question is to know if there are good rings $A$ with 
$\Pic ({\cal 
O}_{S(a,b)}(S(a,b)))$ is not a torsion group.

The answer uses (\cite{G}, Corollary 2 p. 118), a 1963 paper where Goldman 
shows the existence of a
Dedekind domain $A$ with $\Z[X]\subset A\subset \Q[X]$, $\frac{A}{\mathfrak M}$ 
finite 
for all maximal ideal $\mathfrak M$ and such that its ideal class group isn't a 
torsion group. Such a ring is a good ring but not a pictorsion ring (Proposition \ref{prop4.8}).
\vskip 1mm \vskip 0pt
\noindent {\bf Acknowledgments.}

We would like to thank Dino Lorenzini for indicating us the notion ``n is the 
stable range of a ring'' and for attracting our attention on the first version of 
\cite{BE}. We thank Qing Liu for showing us the subtleties of 
Picard group. Moreover his remarks on an early version of our paper gave us the 
opportunity to rewrite and simplify some proofs in using cohomological tools. We also warmly thank the referee for many useful suggestions which led to improvements in the exposition.

\vskip 1mm \vskip 0pt
\noindent{\bf Outline of the paper.} 

In section 2, we prove in particular the equivalence of the two notions  ``good rings'' 
and ``having 
enough homogeneous polynomials'' (Theorem \eqref{thm2.1}). Moreover we rephrase this
in terms of sections of the $A$-scheme $\mathbb{P}^n_A$. 
We give also some examples of good rings. 

In section 3, we give a geometric characterization of good points and good rings 
in terms of Picard group.

In section 4, in order to help the reader, we begin by a subsection in which we list the results concerning good rings and we postpone the proofs further. 
We discuss the stability of good rings by inductive limit, product, quotient, integral extension, 
localisation, transfert to a polynomial ring. Moreover we give many examples of good rings or  not good rings. 
We also study the links between good rings and pictorsion (sous-section \ref{ssec4.3}).

\section{Invertible values of homogeneous polynomials and primitive points.}

In the sequel we adopt the terminology introduced in the introduction.

\begin{prop} \label{prop2.1n} Let $A$ be a ring. The 
following properties are equivalent.
\begin{enumerate}[i)]
\item 
The ring $A$ is a good ring,
\item
for all $a\in A$,  
the group $\frac{(\rho_a(A))^\times}{\rho_a(A^\times)}$ is a torsion group 
where $\rho_a:A\to\frac{A}{aA}$ is the natural epimorphism.
\end{enumerate}
\end{prop}

\begin{proof} 
1) We show that {\sl i)} implies {\sl ii)}.

 Let $a\in A$ and $b\in A$ with $\rho_a(b)\in (\frac{A}{aA})^\times$, then 
 there is $a',b'\in A$ with $a'a+b'b=1$ and so
$(a,b)\in A^2$ is primitive point. Then by {\sl i)} $(a,b)$ is a good point and 
so there is an integer $N\geq 1,\ \lambda \in A$ with $b^N+\lambda a\in 
A^\times$, which means that $(\rho_a(b))^N\in
\rho_a(A^\times)$, i.e. {\sl ii)} is satisfied.

2) We show that {\sl ii)} implies {\sl i)}.

Let $(a,b)\in A^2$ a primitive point and $\rho_a:A\to\frac{A}{aA}$, the natural 
epimorphism. Then $\rho_a(b)\in (\frac{A}{aA})^\times$ and by {\sl ii)} there 
is 
an integer $N\geq 1$  with $(\rho_a(b))^N\in \rho_a(A^\times)$. So there is 
$\lambda \in A$ with $b^N+\lambda a\in A^\times$, i.e. $(a,b)$ is a good point 
and {\sl i)} is satisfied.

\end{proof}
\begin{rem}\label{rem2.1}
\begin{enumerate}[1.]

\item
Property {\sl ii)} (Proposition \ref{prop2.1n}) is trivially satisfied when $a=0$ 
or 
$a\in A^\times$.

\item If $A$ is a field, then part {\sl ii)} (Proposition 
\ref{prop2.1n}) is trivially satisfied, so a field is a good 
ring. 

\item Let $A$ be a ring. If for all $a\in A-\{0\}$, 
$(\frac{A}{aA})^\times$ is a finite or a torsion group, then 
part {\sl ii)} (Proposition \ref{prop2.1n}) is satisfied and so $A$ is a good 
ring. 

\item Let $A$ be a ring. If for all $a\in A-\{0\}$, 
$\frac{A}{aA}$
is finite, then $A$ is a good ring. In particular $\Z$ is a good ring.

\item Let $A$ be a ring. If $A$ is a local ring, one can 
easily show that $A$ is a good ring (Proposition \ref{prop4.3}). One can give an
example of a good ring $A$ and $a\in A-\{0\}$ such that $(\frac{A}{aA})^\times$ is 
not a torsion group (for example $A:=\Q [[T]]$, the formal power series ring  
with rational  coefficients  and $a=T$). 
 \end{enumerate}
\end{rem}

\begin{prop}\label{prop2.2}  {\bf On the section associated to a primitive point.}
 Let $p:=(a_0,a_1,...,a_n)\in A^{n+1}$ be a primitive point i.e. 
$a_0A+a_1A+...+a_nA=A$.
\begin{enumerate}
 \item 
Let $\rho: A[X_0,X_1,...,X_n]\to A[T]$ be the $A$-homomorphism with 
$\rho(X_i)=a_iT$ for 
 $0\leq i\leq n$. Then the homomorphism $\rho$ is an epimorphism and its kernel is \newline
 ${\mathfrak P}_p:=\sum_{0\leq i<j\leq n}(a_iX_j-a_jX_i)A[X_0,X_1,...,X_n]$.
 \item 
 The structural morphism $\Proj A[T]\to \Spec A$ is an isomorphism.
 Let $\mathbb{P}^n_A:=\Proj A[X_0,X_1,...,X_n]$,
and 
 $\pi: \mathbb{P}^n_A\to \Spec A$ the structural morphism, then
 $\rho$ induces a section $\sigma_p: \Spec A\to \mathbb{P}^n_A$ (i.e. 
$\pi\sigma_p=Id_{\Spec A}$) and $\sigma_p(\Spec A)=V_+({\mathfrak P}_p)$.
 
 More precisely if ${\mathfrak p}\in \Spec A$, then  $\sigma_p({\mathfrak 
p})={\mathfrak p}A[X_0,X_1,...,X_n]+{\mathfrak P}_p$.
 
We call $\sigma_p$, {\bf  the section of $\pi$ associated to the primitive point $p$.}
\item
Let $P\in A[X_0,X_1,...,X_n]$ homogeneous of degree $d\geq 1$. The 
following assertions are equivalent 
 \begin{enumerate}
  \item $V_+(P)\cap V_+({\mathfrak P}_p)=\emptyset$,
  \item $V_+(P)\cap \sigma_p(\Spec A)=\emptyset$,
  \item $P(a_0,a_1,...,a_n)\in A^\times$.
 \end{enumerate}
\end{enumerate}
\end{prop}
\begin{proof} 
\noindent 1. Let $\lambda_0,\lambda_1,...,\lambda_n\in A$ be such that 
 $\lambda_0 a_0+\lambda_1 a_1+...+\lambda_n a_n=1$, then \newline
 $\rho(\lambda_0 X_0+\lambda_1 X_1+...+\lambda_n X_n)=T$ and so $\rho $ is onto.
 
As $\rho(a_iX_j-a_jX_i)=0$ we have ${\mathfrak P}_p\subset \Ker \rho$.
 Let $P(X_0,X_1,...,X_n)\in \Ker \rho$ an homogenous polynomial of degree $d$, 
then we have 
 $P(a_0,a_1,...,a_n)=0$.

 Let $\mathfrak M\subset A$ a maximal ideal. As $a_0A+a_1A+...+a_nA=A$ there is  
 $a_i\notin \mathfrak M$. Then
 $a_i$ is invertible in $A_{\mathfrak M}$ and so $P\in \sum_{0\leq j\leq 
n}(a_iX_j-a_jX_i)A_{\mathfrak M}[X_0,X_1,...,X_n]$. It follows that $(\Ker \rho)_{\mathfrak M}
\subset ({\mathfrak P}_p )_{\mathfrak M}$ for all $\mathfrak M$ and so $(\Ker \rho)
\subset {\mathfrak P}_p $.

 
 \noindent 2. This is  an exercise which is left to the reader.

 \noindent 3. 
 As (b) rephrases (a) in the geometric language it is sufficient to prove (a) equivalent (c).
 
Let us assume that (a) is not verified. Then by 2. there is  a prime ideal ${\mathfrak p}\subset A$ with \newline 
 $P(X_0,X_1,...,X_n) \in {\mathfrak p}A[X_0,X_1,...,X_n]+{\mathfrak P}_p$. One has 
 $P(a_0,a_1,...,a_n)\in {\mathfrak p}$ and so
 $P(a_0,a_1,...,a_n)\notin A^\times$. 
 
 Let us assume now that $P(a_0,a_1,...,a_n)\notin A^\times$, then there is ${\mathfrak p}\in \Spec A$ such that 
 $P(a_0,a_1,...,a_n)\in {\mathfrak p}$. It follows that 
 $P(X_0,X_1,...,X_n)\in {\mathfrak p}A[X_0,X_1,...,X_n]+{\mathfrak P}_p$. 
 \end{proof}
 
 \begin{rem}
  On sections for the morphism $\pi:\mathbb{P}^n_A\to \Spec A$.
  
  We know that such a section is associated to an onto $A$-linear map $f: A^{n+1}\to M$ where $M$ is 
  a locally free rank one $A$-module, (\cite{GD}, Theorem 4.2.4, p. 74).
  The case of sections associated to a primitive point corresponds to the case where $M$ is a free rank one 
  $A$-module.

 \end{rem}

\begin{thm} \label{thm2.1} Let $A$ be a ring. The following 
properties are 
equivalent.
\begin{enumerate}[i)]
\item The ring $A$ is a good ring,
\item the ring $A$ has enough homogeneous polynomials,
\item the ring $A$ has enough homogeneous polynomials in two variables.

\end{enumerate}
\end{thm}

\begin{proof} We show that {\sl i)} implies {\sl ii)} implies {\sl iii)} 
implies {\sl i)}. 
\begin{enumerate}[1)]
\item We show {\sl i)} implies {\sl ii)}. The proof works by induction on 
$k=\card S$. Let $n\geq 1$. 
\vskip 1mm \vskip 0pt
 1.1) If $\card S=1$, then $S=\{p_1=(p_{1,1},p_{2,1},...,p_{n,1})\in A^n\}$ and there 
are $u_1,u_2,...,u_n\in A$ with $\sum_{1\leq j\leq n}u_jp_{j,1}=1$. 

Clearly 
$P(X_1,X_2,...,X_n):=\sum_{1\leq i\leq n}u_iX_i$ works.
\vskip 1mm \vskip 0pt 
 1.2) Let $k\geq 1$ and $S':=\{p_1,p_2,...,p_k\}\subset A^n$,  consisting in $k$ 
primitives points.
 By induction hypothesis there is an homogeneous polynomial 
$P(X_1,X_2,...,X_n)\in A[X_1,X_2,...,X_n]$ of degree $d\geq 1$ with 
$P(p_i):=P(p_{1,i},p_{2,i},...,p_{n,i})\in A^\times$, where 
$p_i:=(p_{1,i},p_{2,i},...,p_{n,i})$ for  $1\leq i\leq k$.
 
 Let $q=(q_1,q_2,...,q_n)\in A^n$ be a primitive point with $q\notin S'$ . We want 
to find \newline $R(X_1,X_2,...,X_n)\in A[X_1,X_2,...,X_n]$,  an homogeneous polynomial 
of degree $d'\geq 1$, with $R(p)\in A^\times$ for all $p\in S'$ and for $p=q$.

 \vskip 1mm \vskip 0pt
 1.2.1) Let $a_{i,j,t}:=p_{i,t}q_j-p_{j,t}q_i$  and 
$A_{i,j,t}(X_1,X_2,...,X_n):=p_{i,t}X_j-p_{j,t}X_i$ for $1\leq t\leq k$. We have 
$A_{i,j,t}(q)=a_{i,j,t}$, $A_{i,j,t}(p_t)=0$.

 Let $\mathfrak A_t:=\sum_{1\leq i,j\leq n}a_{i,j,t}A\subset A$. 
 \vskip 1mm \vskip 0pt
 1.2.2) {\sl We show that $P(q)A+ \mathfrak A_t=A$, for all $1\leq t\leq k$.}
 
 Let us assume there is a maximal ideal $\mathfrak M$ in $A$ with $P(q)\in 
\mathfrak 
M$ and  $\mathfrak A_t\subset \mathfrak M,$ i.e. $a_{i,j,t}\in \mathfrak M$ for all 
$1\leq 
i,j\leq n$.\newline Let $\rho:A\to \frac{A}{\mathfrak M}$ be the natural epimorphism, then  
$\rho(p_{i,t})\rho(q_j)-\rho(p_{j,t})\rho(q_i)=\rho(a_{i,j,t})=0$  for all 
$1\leq i,j\leq n$. This means that the matrix 
 $\left[
\begin{array}{cccc}
 \rho(p_{1,t}) &\rho(p_{2,t}) &...&\rho(p_{n,t}) \\
\rho(q_1)&\rho(q_2)&...&\rho(q_n)
\end{array}
\right]$  has rank $\leq 1$.

As $p_t$ is a primitive point  we have $(\rho(p_{1,t}),\rho(p_{2,t}) 
,...,\rho(p_{n,t}))\neq (0,0,...,0)$  and so there is $\lambda_t \in A$ with  $( 
\rho(q_1),\rho(q_2),...,\rho(q_n))=\rho(\lambda_t)(\rho(p_{1,t}),\rho(p_{2,t}) 
,...,\rho(p_{n,t})).$  Now as $q$ is a primitive point we have 
$\rho(\lambda_t)\neq 0$. 
Moreover $\rho(P(q))=\rho(\lambda_t)^{\deg P}\rho(P(p_t))$ and as $P(p_t)\in 
A^\times$ we get  $\rho(P(q))\neq 0$; a contradiction. It follows that $P(q)A+ 
\mathfrak A_t=A$ for all $1\leq t\leq k$.
\vskip 1mm \vskip 0pt
1.2.3) It follows from 1.2.2) that $1=P(q)a_t+\sum_{1\leq i,j\leq 
n}u_{i,j,t}a_{i,j,t}$ for some $a_t,u_{i,j,t}\in A$. 

Let $B_t(X_1,X_2,...,X_n):=\sum_{1\leq i,j\leq 
n}u_{i,j,t}A_{i,j,t}(X_1,X_2,...,X_n)$, then  $B_t(X_1,X_2,...,X_n)$ is nul or 
homogeneous of degree $1$. 

Moreover, we have $1=P(q)a_t+B_t(q)$ and $B_t(p_t)=0$. Then \newline $1=\prod_{1\leq 
t\leq k}(P(q)a_t+B_t(q))=P(q)a+\prod_{1\leq t\leq k}B_t(q)$, with $a\in A$. 
\newline
It follows that $(\prod_{1\leq t\leq k}B_t(q),P(q))$ is a primitive point in 
$A^2$ and  as $A$ is a good ring (Definition \ref{def.Good}), there is $N\geq 1$ and $\lambda\in A$ with 
$P(q)^N+\lambda \prod_{1\leq t\leq k}B_t(q)=\epsilon \in A^\times$. 
\vskip 1mm \vskip 0pt
1.2.4) Note that if $\alpha\geq 1$, then $P^{\alpha}$ is homogeneous of degree 
$\alpha \deg P$ as $P(p_t)\in A^\times$ which prevent $P$ to be a nilpotent 
element in $A[X_1,X_2,...,X_n]$. It follows that up to changing $P$ to
$P^{\alpha}$, we can assume that $N\deg P\geq k$. 

As $q=(q_1,q_2,...,q_n)$  is a primitive point, there is $u_1,u_2,...,u_n\in A$ 
with $\sum_{1\leq s\leq n}u_sq_s=1$. \newline Let $W(X_1,X_2,...,X_n):=\sum_{1\leq s\leq 
n}u_sX_s$ and \newline
$R(X_1,X_2,...,X_n):=P(X_1,X_2,...,X_n)^N+
 \lambda (\prod_{1\leq t\leq 
k}B_t(X_1,X_2,...,X_n))W(X_1,X_2,...,X_n)^{N\deg P-t}$. 
Then $R(p_t)=P(p_t)^N\in 
A^\times$, in particular $R(X_1,X_2,...,X_n)$ is not $0$ and with 1.2.3), \newline
$\lambda (\prod_{1\leq t\leq 
k}B_t(X_1,X_2,...,X_n))W(X_1,X_2,...,X_n)^{N\deg P-t}$  is nul or 
homogeneous of degree $N\deg P$, so $R(X_1,X_2,...,X_n)$
is homogeneous of degree $N\deg P$.\newline  Moreover  $R(q)=P(q)^N+\lambda \prod_{1\leq 
t\leq k}B_t(q)=\epsilon \in A^\times$. This shows {\sl ii)}.

\item The implication {\sl ii)} implies {\sl iii)} follows from the definition.

\item We show {\sl iii)} implies {\sl i)}.

Let us assume that {\sl i)} isn't satisfied, we show that {\sl iii)} isn't 
satisfied.

So there is $(a,b)\in A^2$ a primitive point which isn't a good point, i.e.  
for all $N\geq 1$ and 
$\lambda \in A$ one has $b^N-\lambda a\notin A^\times$. 

Let assume there is an 
homogeneous polynomial $P(X_1,X_2)\in A[X_1,X_2]$ of degree $d\geq 1$ such that 
$P(0,1)=:\epsilon_1\in A^\times$ and $P(a,b)=:\epsilon_2\in A^\times$. We write 
$P(X_1,X_2)=a_0X_2^d+a_1X_1X_2^{d-1}+...+a_dX_1^d$ then  
$a_0=P(0,1)=\epsilon_1\in 
A^\times$ 
and $\epsilon_2=\epsilon_1b^d+\mu a$ where $\mu\in A$.  \newline It follows that
$b^d+(\epsilon_1)^{-1}\mu a=\epsilon_2(\epsilon_1)^{-1}\in A^\times$, which 
gives 
a contradiction. 
\end{enumerate}
\end{proof}

\begin{rem}\label{rem2.2} In (\cite{BE}, Theorem 0.3),  the authors show that PID (principal 
ideal domain) such that the quotients by maximal ideal are finite, have enough 
homogeneous polynomials.  

Our Theorem \ref{thm2.1} with Proposition \ref{prop2.1n}, 
gives a characterization of good rings in terms of their quotient rings by 
principal ideals. With this tool we are able to give in Section 4, numerous 
examples of rings which are or aren't good rings.
\end{rem}

 Now we can rephrase the fact that a ring has enough homogeneous polynomials  in terms of sections associated 
 to primitive points (Proposition \ref{prop2.2} part 2.)  and so we can characterize good rings in terms of an avoidance property 
 (compare with \cite{GLL}, Theorem 5.1, p. 1188).
 
 \begin{thm}\label{thm2.2}
  Let $A$ be a ring, $\mathbb{P}^n_A:=\Proj (A[X_0,X_1,...,X_n])$. Then the following assertions are equivalent.
  \begin{enumerate}
   \item 
   The ring $A$ is a good ring,
   \item 
   the ring $A$ has enough homogeneous polynomials,
   \item 
   for any finite family $\{p_1,p_2,...,p_s\}$ of primitive points in $A^{n+1}$ there is an homogeneous  polynomial 
   $P(X_0,X_1,...,X_n)\in A[X_0,X_1,...,X_n]$ with $\deg P\geq 1$ such that \newline $V_+(P)\cap \sigma_{p_i}(\Spec A)=\emptyset$ for $1\leq i\leq s$
   and $\sigma_{p_i}$ is the section associated to $p_i$.

  \end{enumerate}

 \end{thm}

\section{Primitive points in $A^2$ and Picard group}

\begin{prop} \label{prop3.1}
Let $A$ be a ring and $(a,b)\in A^2$ a primitive point. Let 
$A[x,y]:=\frac{A[X,Y]}{X(aY-bX)A[X,Y]}$ where $x$ (resp. $y$) is the image of 
$X$ (resp. $Y$). Moreover $A[x,y]$ is endowed with the induced grading of 
$A[X,Y]$. Let $S(a,b):=\Proj(A[x,y])$. 
\begin{enumerate}[1.]
\item
The $A$-algebra ${\cal O}_{S(a,b)}(S(a,b))$ is a  free  $A$-module of rank two.

More concretely, there is $\theta \in {\cal O}_{S(a,b)}(S(a,b))$ with 
$\theta_{|D_+(x)}=0$, 
$\theta_{|D_+(y)}=\frac{ay-bx}{y}$. One has $\theta^2=a\theta$, ${\cal 
O}_{S(a,b)}(S(a,b))\simeq \frac{A[T]}{T(a-T)A[T]}$ and $(1,\theta)$ is a basis 
for the 
$A$-module ${\cal O}_{S(a,b)}(S(a,b))$. 

Moreover the scheme $S(a,b)$ is affine and 
isomorphic to $\Spec ({\cal O}_{S(a,b)}(S(a,b)))$.

\item
Let $d\in \N^{>0}$. When considering the $ {\cal O}_{S(a,b)}(S(a,b))$-module 
$ {\cal O}_{S(a,b)}(d)(S(a,b))$ as an $A$ module, we have
 $$ {\cal O}_{S(a,b)}(d)(S(a,b))=\sum_{0\leq k \leq 
d}Ax^ky^{d-k}.$$
\item
Let $a',b'\in A$ with $aa'+bb'=1$. There is an epimorphism 
$u: A[x,y]\to \frac{A[x,y]}{(ay-bx)A[x,y]}\simeq A[a'X+b'Y]$
which is defined by 
$u(x)=a(a'X+b'Y)$ and  $u(y)=b(a'X+b'Y)$.

Let $\Delta(a,b):=\Proj (A[a'X+b'Y])$, $d\geq 1$. Then $u$ induces an 
epimorphism \newline $u':{\cal O}_{S(a,b)}(d)(S(a,b))\to {\cal 
O}_{\Delta(a,b)}(d)(\Delta(a,b))$
such that  $u'(P(x,y))=P(a,b)(a'X+b'Y)^d$ where $P(X,Y)\in A[X,Y]$ is homogeneous 
of degree 
$d$ and  $\{(a'X+b'Y)^d\}$ is a basis for the $A$-module 
${\cal O}_{\Delta(a,b)}(d)(\Delta(a,b))$. 

Analogously,  let $v: A[x,y] \to A[Y]$ with $v(x)=0$ and $v(y)=Y$ and if 
$\Delta(0,1):=\Proj (A[Y])$ 
and  $d\geq 1$ one has an epimorphism 
$v':{\cal O}_{S(a,b)}(d)(S(a,b))\to {\cal O}_{\Delta(0,1)}(d)(\Delta(0,1))$ 
with \newline 
$v'(P(x,y))=P(0,1)Y^d$ where $P(X,Y)\in A[X,Y]$ is homogeneous of degree 
$d$ and $\{Y^d\}$ is a basis for the $A$-module 
${\cal O}_{\Delta(0,1)}(d)(\Delta(0,1))$.

\end{enumerate}
\end{prop}
\begin{proof}

\noindent 1)
Let $\mathbb{P}^1=\mathbb{P}^1_A:=\Proj(A[X,Y])$, and ${\cal J}$ be the coherent 
sheaf of ideals on $\mathbb{P}^1$ defined by the homogeneous polynomial 
$X(aY-bX)$. 

Easily we see that ${\cal J}$ is isomorphic to the twisted sheaf 
${\cal O}_{\mathbb{P}^1}(-2)$. 

Now, we have 
the exact sequence of sheaves on $\mathbb{P}^1$,
\begin{eqnarray}
 \{0\}\to {\cal J}\to {\cal O}_{\mathbb{P}^1}\to \frac{{\cal 
O}_{\mathbb{P}^1}}{{\cal J}}
 \to \{0\}. \label{(2-prop3.1.1)}
\end{eqnarray}
and so
\begin{eqnarray}
 \{0\}\to {\cal J}\to {\cal O}_{\mathbb{P}^1}\to {\cal O}_{S(a,b)}
 \to \{0\}. \label{(3-prop3.1.1)}
\end{eqnarray}
Then~\eqref{(3-prop3.1.1)} gives the long exact cohomological sequence
\begin{eqnarray} \{0\}\to H^0(\mathbb{P}^1,{\cal J})\to H^0(\mathbb{P}^1,{\cal 
O}_{\mathbb{P}^1})
 \to 
  \to H^0(\mathbb{P}^1,{\cal O}_{S(a,b)})\overset{\partial}{\rightarrow}
  H^1(\mathbb{P}^1,{\cal J})\to H^1(\mathbb{P}^1,{\cal O}_{\mathbb{P}^1}). 
\label{(4-prop3.1.1)}
\end{eqnarray}

As ${\cal J}\simeq {\cal O}_{\mathbb{P}^1}(-2)$, it follows from \cite{L}, Lemma 
3.1 p. 195,
that \newline $H^0(\mathbb{P}^1,{\cal J})\simeq H^0(\mathbb{P}^1,{\cal 
O}_{\mathbb{P}^1}(-2))=\{0\},\ 
H^0(\mathbb{P}^1,{\cal O}_{\mathbb{P}^1})= A, $ and 
$H^0(\mathbb{P}^1,{\cal J})\simeq H^0(\mathbb{P}^1,{\cal 
O}_{\mathbb{P}^1}(-2))^\vee\simeq A$.

So we get the exact sequence of $A$-modules, 
\begin{eqnarray}
\{0\}\to A
 \to H^0(\mathbb{P}^1,{\cal O}_{S(a,b)})\overset{\partial}{\rightarrow}
  H^1(\mathbb{P}^1,{\cal J})\to \{0\}. \label{(5-prop3.1.1)}
 \end{eqnarray}

 In particular it follows that 
 $H^0(\mathbb{P}^1,{\cal O}_{S(a,b)})={\cal O}_{S(a,b)}(S(a,b))$ is a 
 free  $A$-module of rank two. 
 
 Moreover looking more precisely 
 at~\eqref{(5-prop3.1.1)} on the cover $\{D_+(X),\ D_+(Y)\}$ of $\mathbb{P}^1$, 
one shows
 that $\partial(\theta)$ is the generator $\frac{X(aY-bX)}{XY}$ of the 
$A$-module
 $H^1(\mathbb{P}^1,{\cal J})$. This shows that $(1,\theta)$ is a basis for the 
$A$-module ${\cal O}_{S(a,b)}(S(a,b))$. 

\vspace{1\baselineskip}
\noindent{\sl We show that $S(a,b)$ is finite over $\Spec A$ and so 
$S(a,b)\simeq \Spec {\cal O}_{S(a,b)}(S(a,b)).$}
\vskip 1mm \vskip 0pt
It suffices to remark 
that the canonical morphism $S(a,b)\to \Spec A$ is quasi-finite, more precisely for all 
prime ideal $\mathfrak p\subset A$, there is two prime homogeneous ideal $\mathfrak P_1, \mathfrak P_2\subset 
A[x,y]$ 
with $\mathfrak P_i\cap A=\mathfrak p$ and $xA[x,y]+yA[x,y]\nsubseteq\mathfrak P_i$ 
for $i=1,2$. Namely we use \cite{L}, 
ex. 
4.2. p. 155 in order to show that  $S(a,b)\to \Spec A$ is finite. Moreover we 
know by part 1. that  ${\cal 
O}_{S(a,b)}(S(a,b))$ is a finite $A$-module and so $S(a,b)\simeq 
\Spec {\cal 
O}_{S(a,b)}(S(a,b)).$ 

Note that we could as well remark that if $\pi :S(a,b)\to \Spec A$ is the structural morphism then 
$\pi^{-1}(D(a))\to D(a)$ (resp. $\pi^{-1}(D(b))\to D(b)$) is finite and as $\Spec A=D(a)\cup D(b)$
the result follows.
\vskip 1mm \vskip 0pt
\noindent 2) 
As ${\cal J}\simeq {\cal O}_{\mathbb{P}^1}(-2)$, ~\eqref{(3-prop3.1.1)} gives 
$$\{0\}\to {\cal O}_{\mathbb{P}^1}(-2)\to {\cal O}_{\mathbb{P}^1}\to {\cal 
O}_{S(a,b)}
 \to \{0\}. $$
 
Let $d\geq 1$, then after tensorizing by ${\cal O}_{\mathbb{P}^1}(d)$ we get 

$$ \{0\}\to {\cal O}_{\mathbb{P}^1}(d-2)\to {\cal O}_{\mathbb{P}^1}(d)\to {\cal 
O}_{S(a,b)}(d)
 \to \{0\}. $$
 
The corresponding long cohomological exact sequence gives in particular  the 
following small exact 
sequence 
\begin{eqnarray}
 H^0(\mathbb{P}^1,{\cal O}_{\mathbb{P}^1}(d))
 \overset{w}{\rightarrow} H^0(\mathbb{P}^1,{\cal O}_{S(a,b)}(d))
 \to H^1(\mathbb{P}^1,{\cal O}_{\mathbb{P}^1}(d-2)) \label{(1-prop3.1.2)}
\end{eqnarray}
where $w(P(X,Y))=P(x,y)$.

As $d\geq 1$, it follows from \cite{L}, Lemma 3.1 p. 195 that 
$H^1(\mathbb{P}^1,{\cal O}_{\mathbb{P}^1}(d-2))=\{0\}$, which shows that the 
homomorphism 
$w:H^0(\mathbb{P}^1,{\cal O}_{\mathbb{P}^1}(d))
 \to H^0(\mathbb{P}^1,{\cal O}_{S(a,b)}(d))$ is an epimorphism; it follows that 
 $$H^0(\mathbb{P}^1,{\cal O}_{S(a,b)}(d))={\cal 
O}_{S(a,b)}(d)(S(a,b))=\sum_{0\leq k \leq 
d}Ax^ky^{d-k}.$$
\vskip 1mm \vskip 0pt
\noindent 3) 
Let $X':=aY-bX$ and $Y':=a'X+b'Y$. As $aa'+bb'=1$, it follows that 
$\varphi: A[X,Y]\to A[X,Y]$  with $\varphi(P(X,Y))=P(X',Y')$ is an 
$A$-automorphism. 
In particular $X',Y'$ are $A$-algebraically independant and $A[X,Y]=A[X',Y']$. 
Moreover
if $P(X,Y)=Q(X',Y')\in A[X,Y]=A[X',Y']$ then $P(X,Y)$ is homogeneous of degree 
$d$ if 
and only if $Q(X',Y')$ is homogeneous of degree $d$.

Let $u'':A[X,Y]=A[X',Y']\to A[Y']$ be the $A$-homomorphism defined by $u''(X')=0$ 
and \newline $u''(Y')=Y'$, then 
$u''$ induces an epimorphism still denoted $u'':{\cal O}_{\mathbb{P}^1}\to {\cal 
O}_{\Delta(a,b)}.$

Let ${\cal G}=\Ker u''$ be the ideal sheaf on $\mathbb{P}^1$ associated to 
$(aY-bX)A[X,Y]$. We have so the exact sequence 
\begin{eqnarray}
 \{0\}\to {\cal G}\to {\cal O}_{\mathbb{P}^1}\to {\cal O}_{\Delta(a,b)}\to 
\{0\}, \label{(1-prop3.1.3)}
\end{eqnarray}
where ${\cal G}(D_+(X'))=\frac{X'}{X'}{\cal O}_{\mathbb{P}^1}(D_+(X'))={\cal 
O}_{\mathbb{P}^1}(D_+(X')$ and
 ${\cal G}(D_+(Y')=\frac{X'}{Y'}{\cal O}_{\mathbb{P}^1}(D_+(Y')).$

We see that ${\cal G}$ is isomorphic to ${\cal O}_{\mathbb{P}^1}(-1)$ and 
tensorizing~\eqref{(1-prop3.1.3)} 
by ${\cal O}_{\mathbb{P}^1}(d)$, we get the exact sequence
\begin{eqnarray}
 \{0\}\to {\cal O}_{\mathbb{P}^1}(d-1)\to {\cal O}_{\mathbb{P}^1}(d)\to 
 {\cal O}_{\Delta(a,b)}(d)\to \{0\}. \label{(2-prop3.1.3)}
\end{eqnarray}
Hence we have the long exact sequence
$$
 \{0\}\to {\cal O}_{\mathbb{P}^1}(d-1)(\mathbb{P}^1)\to {\cal 
O}_{\mathbb{P}^1}(d)(\mathbb{P}^1)
 \overset{u''_d}{\rightarrow}
 {\cal O}_{\Delta(a,b)}(d)(\Delta(a,b))\to H^1(\mathbb{P}^1,{\cal 
O}_{\mathbb{P}^1}(d-1))$$
and as for $d\geq 1$, $H^1(\mathbb{P}^1,{\cal O}_{\mathbb{P}^1}(d-1))=\{0\}$ by 
\cite{L}, Lemma 3.1 p. 195, we get that \newline  $u''_d:{\cal 
O}_{\mathbb{P}^1}(d)(\mathbb{P}^1)
 \to
 {\cal O}_{\Delta(a,b)}(d)(\Delta(a,b))$ is an epimorphism.
 
 \vskip 1mm \vskip 0pt
 \noindent{\sl We describe the map $u''_d$. }
 
 \vskip 1mm \vskip 0pt
 We have the cover $\mathbb{P}^1=D_+(X')\cup D_+(Y')$, $\Delta(a,b)=V_+(X')$ and 
so 
 $D_+(X')\cap \Delta(a,b)=\emptyset$ and $D_+(Y')=\Delta(a,b)$. 
 
 It is sufficient to describe the map 
 $u''_d:{\cal O}_{\mathbb{P}^1}(d)(D_+(Y'))\to {\cal 
O}_{\Delta(a,b)}(d)(\Delta(a,b))=AY'^d,$ 
where $\{Y'^d\}$
 is a basis of ${\cal O}_{\Delta(a,b)}(d)(\Delta(a,b))$.
 
 Let $\frac{Q(X',Y')}{Y'^m}\in {\cal O}_{\mathbb{P}^1}(D_+(Y'))$ with $Q(X',Y')$ 
homogeneous of degree $m+d$.
 
 By definition of $u''$, we have  $u''_d(\frac{Q(X',Y')}{Y'^m})=Q(0,1)Y'^d$ and 
if $P(X,Y):=Q(X',Y')$, we have 
 $Q(0,1)=P(a,b)$. 
 
 It follows that for $P(X,Y)\in {\cal O}_{\mathbb{P}^1}(d)(\mathbb{P}^1)$ we get 
$u''_d(P(X,Y))=
 P(a,b)Y'^d=P(a,b)(a'X+b'Y)^d.$
 
\vskip 1mm \vskip 0pt 
\noindent {\sl Now we describe the map $u'$.}
 
 Let $w:{\cal O}_{\mathbb{P}^1}(d)(\mathbb{P}^1) {\rightarrow}{\cal 
O}_{S(a,b)}(d)(S(a,b))$ defined in \eqref{(1-prop3.1.2)}. 
As $\Ker w\subset \Ker u''_d$,
there is \newline $u':{\cal 
O}_{S(a,b)}(d)(S(a,b))
 {\rightarrow}
 {\cal O}_{\Delta(a,b)}(d)(\Delta(a,b))$ such that $u''_d=u'w.$
 Moreover, as $w$ is an epimorphism by~\eqref{(1-prop3.1.2)}, we deduce that 
 $u'(P(x,y))= P(a,b)(a'X+b'Y)^d$ for $P(X,Y)\in A[X,Y]$ 
 homogeneous of degree $d$. 
 
 Now let us consider $\Delta(0,1):=\Proj(A[Y])$ and $v'': A[X,Y]\to A[Y]$ the 
$A$-morphism defined 
 by $v''(X)=0$ and $v''(Y)=Y$. By an analogous method to the preceeding proof,  
the morphism $v''$ for 
 $d\geq 1$  induces an epimorphism $v''_d: {\cal 
O}_{\mathbb{P}^1}(d)(\mathbb{P}^1)
 \to
 {\cal O}_{\Delta(0,1)}(d)(\Delta(0,1))$ defined  by $v''_d(P(X,Y))=P(0,1)Y^d$.
  It follows that $v':{\cal O}_{S(a,b)}(d)(S(a,b))\to {\cal 
O}_{\Delta(0,1)}(d)(\Delta(0,1))$ is onto
  and defined by  $v'(P(x,y))=P(0,1)Y^d$.
\end{proof} 

\begin{thm} \label{thm3.1}
Let $A$ be a ring, $(a,b)\in A^2$ a primitive point. 

Let $A[x,y]:=\frac{A[X,Y]}{X(aY-bX)A[X,Y]}$ where $x$ (resp.$y$) is the image 
of $X$ (resp. $Y$) by the natural epimorphism. Moreover, $A[x,y]$ is endowed 
with the induced grading.
Let $S(a,b):=\Proj A[x,y]$. The following properties are equivalent.
\begin{enumerate}[i)]
\item  The ${\cal O}_{S(a,b)}(S(a,b))$-module ${\cal O}_{S(a,b)}(1)(S(a,b))$ is 
a torsion element in the Picard group of \newline ${\cal O}_{S(a,b)}(S(a,b)),$
\item  there exists $P(X,Y)\in A[X,Y]$ homogeneous of degree $d\geq 1$ with 
$P(0,1),\ 
P(a,b)\in A^{\times},$
\item   the point $(a,b)\in A^2$ is a good point.
\end{enumerate}
\end{thm}

\begin{proof} 
In order to simplify the notations we shall write $S$ for $S(a,b)$. 

\vspace{1\baselineskip}
\noindent 1){\sl We show i) implies ii).}

It follows from {\sl i)} there is $d\geq 1$ such that  ${\cal 
O}_S(d)(S)$ is  a free rank one ${\cal O}_S(S)$-module and by Proposition 
\ref{prop3.1} 
part 2. there is $P(X,Y)\in A[X,Y]$ homogeneous of degree $d$ such that 
$\{P(x,y)\}$ 
is a basis. 

Then by Proposition \ref{prop3.1} part 3, we have
$u'(P(x,y))=P(a,b)(a'X+b'Y)^{d}$ \newline (resp. $v'(P(x,y))=P(0,1)(Y)^{d}$) is  a basis 
for $A(a'X+b'Y)^{d}$ (resp.for $A(Y)^{d}$),
in other words \newline $P(a,b)\in A^\times$ (resp. $P(0,1)\in A^\times$). 

\vspace{1\baselineskip}
\noindent 2){\sl We show ii) implies i).}
\begin{enumerate}
 \item
{\sl Let $P(X,Y)\in A[X,Y]$ homogeneous of degree $d\geq 1$ with 
$P(0,1),\ 
P(a,b)\in A^{\times}$, we show that in $\Proj (A[X,Y])$ we  have  $V_+(P(X,Y))\cap 
V_+(X(aY-bX))=\emptyset$.}

First we consider $V_+(P(X,Y))\cap V_+(X)$. Let $\mathfrak P\subset A[X,Y]$ be 
an homogeneous ideal with 
$X\in  \mathfrak P$ and $XA[X,Y]+YA[X,Y]\nsubseteq \mathfrak P.$ We have 
$P(X,Y)=a_0X^d+a_1X^{d-1}Y+...+a_dY^d$ with $d\geq 1$ and $a_d=P(0,1)\in 
A^\times$.
If $P\in \mathfrak P$, then $a_dY^d\in \mathfrak P$ and so $Y\in \mathfrak P$, a 
contradiction. 
So $V_+(P(X,Y))\cap V_+(X)=\emptyset.$

Now we consider $V_+(P(X,Y))\cap V_+(aY-bX)$.  Let $a',b'\in A$ with 
$aa'+bb'=1$, we can write 
$P(X,Y)=b_0(a'X+b'Y)^d +b_1(a'X+b'Y)^{d-1}(aY-bX)+...+b_d(aY-bX)^d$  
with $b_i\in 
A$. As $P(a,b)\in A^\times$ 
it follows that $b_0\in A^\times$.  Let $\mathfrak P\subset A[X,Y]$ be an 
homogeneous ideal with 
$aY-bX\in  \mathfrak P$ and $XA[X,Y]+YA[X,Y]\nsubseteq \mathfrak P$ and 
$P(X,Y)\in \mathfrak P$, 
then $a'X+b'Y\in \mathfrak P$, but \newline 
$(aY-bX)A[X,Y]+(a'X+b'Y)A[X,Y]=XA[X,Y]+YA[X,Y]$, a contradiction.\newline 
So $V_+(P(X,Y))\cap V_+(aY-bX)=\emptyset.$ 

\noindent 
{\sl Now we show that ${\cal O}_{S}(d)(S)$ is a free ${\cal O}_{S}(S)$-module of rank 
one.}

\item
{\sl We show that $P(X,Y)$ doesn't divide zero in $A[X,Y]$.\newline  We have 
$P(X,Y)=a_0X^d+a_1X^{d-1}Y+...+a_dY^d$ with
$a_d=P(0,1)\in A^\times$.}

Let us assume that $Q(X,Y)=q_0X^m+q_1X^{m-1}Y+...+q_kX^{m-k}Y^k\in A[X,Y]$ with 
$q_k\neq 0$ and 
$P(X,Y)Q(X,Y)=0$, then the coefficient of $X^{m-k}Y^{d+k}$ in $P(X,Y)Q(X,Y)$ is 
nul, so 
$a_dq_k=0$, a contradiction.  
\item
As $P(X,Y)$ isn't a zero divisor in $A[X,Y]$, we can consider the 
Cartier divisor in $\mathbb{P}^1_A$ defined by 
$\{(D_+(X),\frac{P(X,Y)}{X^d}),(D_+(Y),\frac{P(X,Y)}{Y^d})\}$.
Let $\cal L$, the invertible sheaf on $\mathbb{P}^1_A$ associated to this 
divisor. Then
${\cal L}(D_+(X))$ (resp. ${\cal L}(D_+(Y))$) is the ${\cal 
O}_{\mathbb{P}^1}(D_+(X))$-free module
with basis $\frac{X^d}{P(X,Y)}$ (resp. ${\cal O}_{\mathbb{P}^1}(D_+(Y))$-free 
module
with basis $\frac{Y^d}{P(X,Y)}$) and 
$\frac{Y^d}{P(X,Y)}=(\frac{Y}{X})^d\frac{X^d}{P(X,Y)}$ with \newline
$(\frac{Y}{X})^d\in {\cal O}_{\mathbb{P}^1}(D_+(XY))^\times.$ 

As the invertible sheaf ${\cal O}_{\mathbb{P}^1}(d)$ is such that ${\cal 
O}_{\mathbb{P}^1}(D_+(X))$ 
(resp. ${\cal O}_{\mathbb{P}^1}(D_+(Y))$) is an ${\cal 
O}_{\mathbb{P}^1}(D_+(X))$-free module 
of basis $X^d$ (resp. the ${\cal O}_{\mathbb{P}^1}(D_+(Y))$-free module 
of basis $Y^d$) and as \newline $(\frac{Y}{X})^d\in {\cal 
O}_{\mathbb{P}^1}(D_+(XY))^\times$ it follows that 
the two sheaves $\cal L$ and ${\cal O}_{\mathbb{P}^1}(d)$ are isomorphic. It 
follows that 
${\cal L}\otimes_{{\cal O}_{\mathbb{P}^1}}{\cal O}_S\simeq {\cal O}_S(d)$. 
\item 
{\sl We show that ${\cal L}\otimes_{{\cal O}_{\mathbb{P}^1}}{\cal O}_S\simeq 
{\cal O}_S$. }

We know that ${\cal L}\otimes_{{\cal O}_{\mathbb{P}^1}}{\cal O}_S(D_+(x))$
(resp. ${\cal L}\otimes_{{\cal O}_{\mathbb{P}^1}}{\cal O}_S(D_+(y))$ is 
generated by $\frac{x^d}{P(x,y)}$ 
(resp. $\frac{y^d}{P(x,y)}$). It is sufficient to prove that 
$\frac{P(x,y)}{x^d}\in {\cal O}_S(D_+(x))^\times$ 
(resp. $\frac{P(x,y)}{y^d}\in {\cal O}_S(D_+(y))^\times$).  As \newline $V_+(P(X,Y))\cap 
V_+(X(aY-bX))=\emptyset$, it follows 
that the set  $V_+(P(x,y))\subset S$ is empty.\newline  So $V(\frac{P(x,y)}{x^d})\cap 
D_+(x)=\emptyset$,
$V(\frac{P(x,y)}{x^d})\cap D_+(y)=\emptyset$. It follows that 
${\cal L}\otimes_{{\cal O}_{\mathbb{P}^1}}{\cal O}_S\simeq {\cal O}_S$.

We have shown that
${\cal O}_S(d)={\cal O}_S(1)^{\otimes d}\simeq {\cal O}_S$. Now as $S$ is affine 
by \ref{prop3.1}, it follows 
that $({\cal O}_S(1)({\cal O}_S(S)))^{\otimes d}\simeq {\cal O}_S(S)$. 

\end{enumerate}
\noindent 3) {\sl We show ii) implies iii).}
\vskip 1mm \vskip 0pt
Let $P(X,Y)=a_0Y^d+a_1XY^{d-1}+...+a_dX^d$, $d\geq 1$ with  $P(0,1)=a_0\in 
A^{\times}$,\newline  $P(a,b)=a_0b^d+a(a_1b^{d-1}+a_2b^{d-2}a+...+a_da^{d-1})\in 
A^{\times}.$ \newline 
Let $\lambda:=(a_0)^{-1}(a_1b^{d-1}+a_2b^{d-2}a+...+a_da^{d-1})\in A$, then 
$b^d+\lambda a\in A^{\times}$, i.e. $(a,b)\in A^2$ is a good point.  
\vskip 1mm \vskip 0pt
\noindent 4){\sl We show iii) implies ii).}

\vskip 1mm \vskip 0pt
As $(a,b)\in A^2$ is a good point, there is $N\geq 1$ and $\lambda\in A$, with 
$b^N+\lambda a\in A^\times$.
Let $a',b'\in A$ with $aa'+bb'=1$ and  $P(X,Y):=Y^N-\lambda X(a'X+b'Y)^{N-1}$. 

Then $P(0,1)=1$ and 
$P(a,b)=b^N-\lambda a\in A^\times$.
\end{proof}
\begin{rem}\label{rem3.1}
If for any primitive point $(a,b)\in A^2$, {\sl iii)} in Theorem 
\ref{thm3.1} is satisfied, then $A$ is a good ring and conversely.

As well if for all primitive point $(a,b)\in A^2$, {\sl i)} in Theorem 
\ref{thm3.1} is satisfied, i.e. the \newline ${\cal 
O}_{S(a,b)}(S(a,b))$-module ${\cal 
O}_{S(a,b)}(1)(S(a,b))$ is 
a torsion element in the Picard group of ${\cal O}_{S(a,b)}(S(a,b))$, then $A$ 
is a good ring and conversely. 
\end{rem} 
\begin{coro}\label{coro3.1}
Let $A$ be a pictorsion ring, then  $A$ is a good ring.
\end{coro}
\begin{proof}
As 
${\cal O}_{S(a,b)}(S(a,b))$ 
is a 
finite $A$-algebra, it follows that $\Pic( {\cal O}_{S(a,b)}(S(a,b)) )$ is a 
torsion group 
and so part {\sl i)} in Theorem \ref{thm3.1} is satisfied. Therefore 
pictorsion rings are good rings.
\end{proof}
Note that the condition to be of pictorsion for $A$ is strong : it would 
suffice that $\Pic( {\cal O}_{S(a,b)}(S(a,b)) )$ be a torsion group. This will 
be study 
 in paragraph 4.3.

\section{Examples of good rings and of not good rings.}

\subsection{Examples of good rings}
In this section we first give a list of good rings (Definition \ref{def.Good}) and most of proofs are postponed to the end of the section.
\subsubsection{Properties of the family of good rings}
    
 \begin{enumerate}
  \item {\sl Stability with inductive limits}
  
A ring $A$ which is the inductive limit of good rings 
$(A_i)_{i\in I}$ is a good ring.

\item {\sl Stability with finite products}

Let $A=A_1\times A_2\times ...\times A_r$ with $A_i$ a good ring for $1\leq 
i\leq r$ then $A$ is a good ring. Namely let $a=(a_1,...,a_r), 
b:=(b_1,...,b_r), 
u=(u_1,...,u_r), v=(v_1,...,v_r)\in A$ with $au+bv=1$. Then $a_iu_i+b_iv_i=1$ 
for 
$1\leq i\leq r$ and there is $N_i\geq 1$ and $\lambda_i\in A_i$ with 
$b_i^{N_i}+\lambda_ia_i=\epsilon_i\in A_i^\times$, then for $N:=\prod_{1\leq 
i\leq r} N_i$ there is 
$\mu_i\in A_i$ and $\epsilon'_i\in A_i^\times$ with $b_i^N+\mu_i 
a_i=\epsilon'_i$, and so $b^N+\mu a=\epsilon'$ where 
$\mu=(\mu_1,...,\mu_r)$  and $\epsilon'=(\epsilon'_1,...,\epsilon'_r)$.

\item {\sl Stability by quotients}

Let $A$ be a good ring and $I\subset A$ an ideal then $\frac{A}{I}$ is a good ring. 
Namely let $\rho: A\to \frac{A}{I}$ be the natural epimorphism and $(\rho(a),\rho(b))\in 
(\frac{A}{I})^2$, a primitive point, then there is $(a',b')\in A^2$ such that $aa'+bb'-1=c\in I$
and so $((aa'-c),b)\in A^2$ is a primitive point. Now as $A$ is a good ring then $((aa'-c),b)$
is a good point (Definition \ref{def.Good}) and so there is $N\geq 1$ and $\lambda\in A$ with 
$b^N+\lambda (aa'-c)=\epsilon\in A^\times$. 
It follows that $(\rho(b))^N+\rho(\lambda a')\in 
\rho(A^\times)\subset (\frac{A}{I})^\times$ and so $\frac{A}{I}$ is a good ring.

The reverse is false in general. For example, let  $A:=\Q[T]$ and  $I=T\Q [T]$, then 
$A$ isn't a good ring (Proposition \ref{lem4.3}) and $\frac{A}{I}=\Q$ is a good ring.

\item {\sl Good rings and the Jacobson radical}
\begin{enumerate}
 \item 
Let $A$ be a ring and $\mathfrak R$ its Jacobson radical i.e. the 
intersection 
of the maximal ideals.  Let $\mathfrak A$ be an 
ideal 
with $\mathfrak A\subset \mathfrak R$. Then $A$ is a good ring
iff $\frac{A}{\mathfrak A}$  is a good ring, (Proposition \ref{prop4.2}).
\item 
Let $A$ be a ring and $\mathfrak R$ its Jacobson radical. If for all $x\in 
A-\mathfrak R$, $(\frac{A}{xA})^\times$ is a torsion group, then $A$ is a good 
ring, (Proposition \ref{prop4.1}).
\end{enumerate}

In conclusion, on one side, in order to show that a ring is good ring, after a quotient we can assume that the 
Jacobson radical is trivial. On the other side, from a given good ring $B$, one can built new good rings such as $A:=\frac{B[X]}{X^nB[X]}$. Namely let  $\mathfrak R$ be the Jacobson radical of $A$ and $x$ the image of $X$ in $A$. Then as $xA\subset \mathfrak R$ and $\frac{A}{xA}\simeq B$  
is a good ring, it follows that $A$ is a good ring.
\end{enumerate}

\subsubsection{Classical examples of good rings}

\begin{enumerate}
 \item 
 If  $A$ is a field it follows from  (Definition \ref{def.Good})
 that $A$ is a good ring. More generally let $A$ be a 
semi-local ring (i.e. $A$ has a finite number of maximal ideals) then 
$A$ is a good ring (Proposition \ref{prop4.3}).
\item 
A PIR (principal ideal ring) such that $\frac{A}{aA}$ is finite for all $a$ 
not a zero divisor, is a good ring (Proposition \ref{prop4.4}).

This generalizes the particular case of PID (principal ideal domain) which is 
considered in (\cite{BE}, Theorem 0.3). Namely, they show that if $A$ is a 
PID such that $\frac{A}{\mathfrak M}$ is finite for all maximal, then $A$ has enough 
polynomial and 
so $A$ is a good ring by Theorem \ref{thm2.1}.
\item
In particular if $A$ is a Dedekind domain such that the residue fields 
$\frac{A}{\mathfrak M}$ are finite for the maximal ideals $\mathfrak M$, then 
for all 
$x\in 
A-\{0\}$, $\frac{A}{xA}$ is finite  and so $A$ is a good ring (Proposition 
\ref{prop4.1}).

\end{enumerate}

\begin{rem}\label{rem4.1.2}
 We note the following 
 \begin{enumerate}
 \item
Let $A$ be a good ring and $S$ be a multiplicative subset. 
A natural question is  if $S^{-1}A$ is a good ring. 

The answer is yes for the classical examples listed above, but we give a counterexample  (Proposition 
\ref{Loc}).  

  \item 
 Let $A$ be a Dedekind domain with zero Jacobson radical. Is $A$ a good ring
 iff for all maximal ideal ${\mathfrak M}$ the residue field 
 $\frac{A}{\mathfrak M}$ is finite?
 
Proposition \ref{propLocPol} with $d=1$ gives a negative answer.

 \end{enumerate}

\end{rem}

\subsubsection{ New examples}

\begin{enumerate}
 \item  {\sl Good rings with Krull dimension $d\geq 1$ and infinite residue fields}
 
Let $p\in \Z$, be a prime, $F:=\{\{0\}\cup \{p^k\ ,k\geq 1\}$, $d\geq 1$ and  $Z:=F^d$. 
\newline Then $S:=\{P\in \Q[X_1,X_2,...,X_d]\ |\ \forall z\in Z,\ P(z)\neq 0\}$ is a multiplicative 
set in 
$\Q[X_1,X_2,...,X_d]$ and $A:=S^{-1}\Q[X_1,X_2,...,X_d]$ is a good ring. Moreover $A$
is noetherian, factorial, with zero Jacobson radical, Krull dimension $d$ and 
the residue fields $\frac{A}{\mathfrak M}$ with $\mathfrak M$ a maximal ideal, are infinite.

Note that in particular when $d=1$, the good ring $A$ is  a PID with 
zero Jacobson radical and  the residue fields $\frac{A}{\mathfrak M}$ with $\mathfrak M$ a maximal ideal, are infinite.

For the proof, see Proposition \ref{propLocPol}

\item  {\sl Good rings and integral extensions}
\begin{enumerate}
\item 
 Let $A$ be a ring and $\rho: \Z\to A$ be the natural homomorphism. If $A$ is 
integral over $\rho (\Z)$, then $A$ is a good ring (Proposition \ref{prop4.5}).
\item
Let $A$ be a ring and $L$ be a sub-field of an algebraic closure of a finite 
field and $\rho: L[T]\to A$ be an homomorphism with $A$ integral over 
$\rho(L[T])$, then $A$ is a good ring (Proposition \ref{prop4.5}).

\end{enumerate}
\item {\sl Good rings inside infinite products of rings}
\begin{enumerate}
\item
Let $(K_i)_{i\in I}$, a family of fields indexed by the set $I$,  then $\prod_{i\in I}K_i$ is a good ring. This follows from Proposition 
\ref{prop2.1n}.
 \item 
Let $A$ be the subring of $\Z^\N$ of stationary sequences.  Let 
$x\in A$ be the sequence $(x_k)_{k\geq 0}$ then $\frac{A}{xA}\simeq \prod_{k\geq 
0}\frac{\Z}{x_k\Z}$. Then $(\frac{A}{xA})^\times$ is a torsion group and 
$A$ is a good ring (Remark \ref{rem2.1} and Proposition \ref{prop2.1n}).
\end{enumerate}
\item 
Let $X$ be a compact topological space, $A:={\cal C}(X,\R)$ be the ring of 
continuous functions on $X$ with real values, then $A$ is a good ring and more 
precisely  if $f,g,u,v\in A$ with $fu+gv=1$, then $f^2+g^2\in A^\times$.
\item
 Let $X$ be an algebraic non singular curve over $\Q$; let us assume that 
$X(\Q)\neq \emptyset$. Let $a\in X(\Q)$ and $Y:=X(\Q)-\{a\}$. 

Let $A:=\{f\in 
K(X)\ |\ v_y(f)\geq 0,\ \forall y\in Y\}$  where $K(X)$ is the field of rational 
functions on $X$ and $v_x$ is the valuation at $x$. Then $A$ is a good ring and
more precisely if 
$u,v,u',v'\in A$ with $uu'+vv'=1$; then $\forall y\in Y$, one has
$v_y(u)=0$ or $v_y(v)$=0 and as the residue field $\frac{{\cal O}_{X,y}}{ {\mathfrak M}_{y} }=\Q$, 
it follows that $v_y(u^2+v^2)=0$
and so
 $u^2+v^2\in A^\times$.
\end{enumerate}

\subsubsection{ Proofs}
\begin{prop} \label{prop4.2}
Let $A$ be a ring, $\mathfrak R$ its Jacobson radical  i.e. 
the intersection  of the maximal ideals and 
$\mathfrak A\subset \mathfrak R$ an ideal.

The 
following properties are equivalent.
\begin{enumerate}[i)]
 \item The ring $A$ is a good ring,
 \item the ring $\frac{A}{\mathfrak A}$ is a good ring.
\end{enumerate}

\end{prop}

\begin{proof}
As i) implies ii) is always satisfied (see 4.1.1 part 3),
we show that ii) implies i).

Let $\rho: A\to \frac{A}{\mathfrak A}$ be the natural epimorphism and $a,b,a',b'\in A$ with $aa'+bb'=1$.\newline  Then $\rho(a )\rho(a' )+\rho( b)\rho( b')=1$. 
As $\frac{A}{\mathfrak A}$ is a good ring, there is $N\geq 1$, $\mu\in A$ with \newline
$\rho(b)^N+\rho(\mu) \rho(a)=\epsilon \in (\frac{A}{\mathfrak A})^\times$. Let 
$e\in 
A$ with $\rho(e)=\epsilon$, then $b^N+\mu a=e+\alpha$ with $\alpha\in \mathfrak 
A$. 
If 
$e+\alpha\notin A^\times$ there is a maximal ideal $\mathfrak M\subset A$ with 
$e+\alpha\in \mathfrak M$. As $\alpha\in \mathfrak R$ then $e\in \mathfrak M$. 
As $\Ker 
\rho\subset \mathfrak R\subset \mathfrak M$, it follows that $\rho(\mathfrak 
M)$ is a 
maximal ideal of $\frac{A}{\mathfrak A}$, but $\epsilon=\rho(e)\in 
\rho(\mathfrak M)$ 
which is in contradiction with $\epsilon \in (\frac{A}{\mathfrak A})^\times$. 
Finally $A$ is a good ring.
\end{proof}

\begin{prop} \label{prop4.1}
Let $A$ be a ring, $\mathfrak R$ its Jacobson radical.  If for all $x\in 
A-\mathfrak R$, $(\frac{A}{xA})^\times$ is a torsion group, then $A$ is a good 
ring.

In particular if $A$ is a Dedekind domain such that the residue fields 
$\frac{A}{\mathfrak M}$ are finite for the maximal ideals $\mathfrak M$, then 
for all 
$x\in 
A-\{0\}$, $\frac{A}{xA}$ is finite  and so $A$ is a good ring.
\end{prop}

\begin{proof}
Let $a,b,u,v\in A$ with $au+bv=1$. 

If $a\in A-\mathfrak R$, then 
$(\frac{A}{aA})^\times$ is a torsion group by hypothesis and so there is $N\geq 
1$ 
and $\lambda\in A$ with $b^N+\lambda a=1$, i.e.  $(a,b)$ is a good point 
(Definition \ref{def.Good}).
 
If $a\in \mathfrak R$, then $1-av\in A^\times$ and so $b\in A^\times$ and 
$b^1+0a=b\in A^\times$; again  $(a,b)$ is a good point. 

It follows that $A$ is a good ring (Definition \ref{def.Good}).

Let us assume now that $A$ is a Dedekind domain such that the residue fields 
$\frac{A}{\mathfrak M}$ are finite for the maximal ideals $\mathfrak M$.

If $x\in A^\times$, then $\frac{A}{xA}=\{0\}$ and $(\frac{A}{xA})^\times$ is a 
torsion group.

If $x\notin A^\times$ and $x\neq 0$, then $xA=\mathfrak M_1^{\alpha_1}\mathfrak 
M_2^{\alpha_2}...\mathfrak M_r^{\alpha_r}$ where $r\geq 1$, $\mathfrak M_i$ a 
maximal 
ideal and $\alpha_i\geq 1$ for $1\leq i\leq r$.  Then $\frac{A}{xA}\simeq 
\frac{A}{\mathfrak M_1^{\alpha_1}}\times \frac{A}{\mathfrak 
M_2^{\alpha_2}}\times 
...\times \frac{A}{\mathfrak M_r^{\alpha_r}}$ and as $|\frac{A}{\mathfrak 
M_i^{\alpha_i}}|= |\frac{A}{\mathfrak M_i}|^{\alpha_i}$ (to see this one can 
replace $A$ by its localisation at $\mathfrak M_i$ which is a principal local 
ring 
and $\mathfrak M_i$ by its maximal ideal)  the ring 
$\frac{A}{xA}$ is finite and 
$(\frac{A}{xA})^\times$ is a torsion group and so $A$ is a good ring (Remark 
\ref{rem2.1} and Proposition \ref{prop2.1n}).

\end{proof}

\begin{prop} \label{propLocPol}
Let $p\in \Z$, be a prime, $F:=\{\{0\}\cup \{p^k\ ,k\geq 1\}$, $d\geq 1$ and $Z:=F^d$. 
Then \newline $S:=\{P\in \Q[X_1,X_2,...,X_d]\ |\ \forall z\in Z,\ P(z)\neq 0\}$ is a multiplicative 
set in 
$\Q[X_1,X_2,...,X_d]$ and \newline $A:=S^{-1}\Q[X_1,X_2,...,X_d]$ is a good ring. Moreover $A$
is noetherian, factorial, with zero Jacobson radical, Krull dimension $d$ and 
the residue fields $\frac{A}{\mathfrak M}$ with $\mathfrak M$ a maximal ideal, are infinite.
\end{prop}
\begin{proof}
We shall write $\X$ for $X_1,X_2,...,X_d$. 

\vskip 1mm \vskip 0pt
0) Let $a\in A$ and $\rho_a:A\to\frac{A}{aA}$ be the natural 
epimorphism. Let $b\in A$ such that $(a,b)\in A^2$ is a primitive point i.e. $\rho_a(b)\in \rho_a(A)^\times$.

{\sl We prove that  $\rho_a(b)\in \rho_a(A^\times)$ i.e. $(a,b)\in A^2$ 
is a good point and so $A$ is a good ring. }
\vskip 1mm \vskip 0pt
1) One can write $a=\frac{P(\X)}{T(\X)}$, $b=\frac{Q(\X)}{T(\X)}$ 
with $P,Q,T\in \Z[\X ]$ and $T\in S$. There is $U,V\in \Z[\X]$ and 
$R\in S$ with
\begin{eqnarray}
 P(\X)U(\X)+Q(\X)V(\X)=R(\X). \label{1LocPol}
\end{eqnarray}

In particular $\forall z\in Z$, one has $R(z)\neq 0$. 
\vskip 1mm \vskip 0pt
2) {\sl There is  $M\geq 1$ such that for all $z\in Z$, one has 
$(P(z),Q(z))\not\equiv 0$ modulo $p^M\Z^2$.}

Let $\Z_p$ be the completion of $\Z$ for the $p$-adique absolute $|.|_p$. Then $\Z_p$ is
compact and  if $F\subset \Z_p$ is  closed, it is compact. Moreover  $Z=F^d$ is compact.
The map $z\in Z\to |f(z)|_p\in \R$ is continuous and as $R(z)\neq 0$ for $z\in Z$, there is $M\geq 1$ with
$R(z)\not\equiv 0 \mod p^M\Z.$

From relation \eqref{1LocPol} we deduce that if $P(z)\equiv 0 \mod p^M\Z$,
(resp. $Q(z)\equiv 0 \mod p^M\Z$ ) then \newline
$Q(z)\not\equiv 0 \mod p^M\Z$, (resp. $P(z)\not\equiv 0 \mod p^M\Z$). 
\vskip 1mm \vskip 0pt
3) {\sl We show that there is $(s,t)\in \Z^2$ such that $\forall z\in Z$, one has  $P(z)t-Q(z)s\neq 0$.}
\vskip 1mm \vskip 0pt
3.1) Let $z\in Z$ and $C(P(z),Q(z)):=\{(u,v)\in \Z^2\ |\ P(z)v-Q(z)u=0\}$. 

Let $N\geq 1$ and $\pi_N:\Z\to \frac{\Z}{p^N\Z}$ be the natural epimorphism and \newline
$\Theta(z):= \{(x,y)\in (\frac{\Z}{p^N\Z})^2\ |\ \pi_N(P(z))y-\pi_N(Q(z))x=0\}$, then 
\begin{eqnarray}
 \pi_N\times \pi_N(C(P(z),Q(z))\subset \Theta(z). \label{2LocPol}
\end{eqnarray}
We remark that if $z'\equiv z$ modulo $p^N\Z^d$, then 
\begin{eqnarray}
 \pi_N\times \pi_N(C(P(z'),Q(z'))\subset \Theta(z). \label{3LocPol}
\end{eqnarray}
Let $Z_N:=\{0\}\cup\{p,p^2,...,p^{N-1}\}$, it follows from \eqref{2LocPol} and \eqref{3LocPol} that 
\begin{eqnarray}\pi_N\times \pi_N(\cup_{z\in Z}C(P(z),Q(z))\subset
 \cup_{z\in Z_N}\{(x,y)\in (\frac{\Z}{p^N\Z})^2\ |\ 
\pi_N(P(z))y-\pi_N(Q(z))x=0\} . \label{4LocPol}
\end{eqnarray}

3.2) We assume that $N\geq M$ where $M$ is defined in 2). Then $\forall z\in Z$, we have \newline
$\card \{(x,y)\in (\frac{\Z}{p^N\Z})^2\ |\ \pi_N(P(z))y-\pi_N(Q(z))x=0\}\leq p^{N+M}$, and so there 
is $(a,b)\in \Z^2$, such that $\forall z\in Z$, we have $P(z)b-Q(z)a\neq 0$. 

We assume that $\pi_N(P(z))y-\pi_N(Q(z))x=0$. We can write  $\pi_N(P(z))=\pi_N(p^\alpha)u$ and \newline
$\pi_N(Q(z))=\pi_N(p^\beta)v$ with $u,v\in \frac{\Z}{p^N\Z}^\times$. Moreover we can assume that 
$0\leq \alpha\leq \beta\leq N$.

Now we remark that 2) implies
\begin{eqnarray}
 \alpha\leq M. \label{5LocPol}
\end{eqnarray}
So we can write 
\begin{eqnarray}
 \pi_N(p^\alpha)u(y-\pi_N(p^{\beta-\alpha})u^{-1}vx)=0. \label{6LocPol}
\end{eqnarray}
It follows that for $x\in \frac{\Z}{p^N\Z}$ there are $p^\alpha$ solutions $y$ to equation \eqref{6LocPol}.
This shows
\begin{eqnarray}
\card \{(x,y)\in (\frac{\Z}{p^N\Z})^2\ |\ \pi_N(P(z))y-\pi_N(Q(z))x=0\}\leq p^{N+M}. \label{7LocPol}
\end{eqnarray}
Now it follows from \eqref{4LocPol}, \eqref{5LocPol}, \eqref{6LocPol}, \eqref{7LocPol} that
\begin{eqnarray}
\card \{\pi_N\times \pi_N(\cup_{z\in Z}C(P(z),Q(z))\}\leq N^dp^{N+M}. 
\end{eqnarray}
Now for $N$ big enough, we have $N^dp^{N+M}<p^{2N}$ and so there is $(s,t)\in \Z^2$ such that
\begin{eqnarray}
\pi_N\times \pi_N(s,t)\notin\pi_N\times \pi_N(\cup_{z\in Z}C(P(z),Q(z)). 
\end{eqnarray}
and so 
\begin{eqnarray}
(s,t)\notin\cup_{z\in Z}C(P(z),Q(z)). 
\end{eqnarray}
i.e. $\forall z\in Z$, we have $P(z)t-Q(z)s\neq 0.$ 
\vskip 1mm \vskip 0pt
4) {\sl We prove that  $\rho_a(b)\in \rho_a(A^\times)$ i.e. $(a,b)\in A^2$ 
is a good point and so $A$ is a good ring. }

We have $a=\frac{P(\X)}{T(\X)}$, $b=\frac{Q(\X)}{T(\X)}$ 
with $P,Q,T\in \Z[\X ]$ and $T\in S$.

By 3), $\forall z\in Z$, one has $P(z)t-Q(z)s\neq 0$ and so $a(z)t-a(z)s\neq 0$ i.e. 
$a(\X)t-b(\X)s\in S\subset A^\times$.

If $s\neq 0$ then $b-\frac{t}{s}a\in A^\times$ and so $\rho_a(b)\in \rho_a(A^\times)$.

If $s=0$, then $a\in S$ and $\rho_a(A)=\{0\}$ and again $\rho_a(b)\in \rho_a(A^\times)$.
\vskip 1mm \vskip 0pt
5) As $\Q[\X]$ is integral, noetherian and factorial , this is the same for $A=S^{-1}\Q[\X]$. 
\vskip 1mm \vskip 0pt
6) {\sl We show that $\dim A=d$. }

As $\dim \Q[\X]=d$, it follows that $\dim A\leq d$. 

We show that $\{0\}\subset X_1A\subset  X_1A+ X_2A\subset ...\subset  X_1A+ X_2A+...+X_dA$ is a chain 
of prime ideals of length $d$. 

Let $k\geq 1$ and $r_k:\Q[\X]\to \Q[X_{k+1},X_{k+2},...,X_d]$ the homomorphism with \newline
$r_k(P(\X)):=P(0,...,0,X_{k+1},X_{k+2},...,X_d)$.
If $T\in S$, $r_k(T)=r_k(0,...,0,X_{k+1},X_{k+2},...,X_d)$. \newline If
$z_{k+1},z_{k+2},...,z_d\in F$, then $(0,...,0,z_{k+1},z_{k+2},...,z_d)\in Z$, so  
$r_k(0,...,0,z_{k+1},z_{k+2},...,z_d)\neq 0$, and  $r_k(T)\neq 0$.

It follows that $r_k$ extends to an homomorphism 

$r'_k:S^{-1}\Q[\X]\to \Q(X_{k+1},X_{k+2},...,X_d)$, so $\Ker r'_k=X_1A+ X_2A+...+X_kA$ is a prime ideal. 

As $r_k(X_{k+1})=X_{k+1}\neq 0$, we get that $\Ker r'_k\neq \Ker r'_{k+1}$. 
\vskip 1mm \vskip 0pt
7) {\sl We show that the Jacobson radical of  the ring $A$ is zero.}

Let $z\in Z$ and $f_z:\Q[\X]\to \Q$, defined by $f_z(P):=P(z)$. Then $f_z$ extends to an homomorphism 
$f'_z:S^{-1}\Q[\X]\to \Q$ and ${\mathfrak M}_z:=\Ker f'_z\subset A$ is a maximal ideal. Let $\frac{P(\X)}{T(\X)}
\in A
$ 
with $P(\X)\in \Q[\X]$ and $T(\X)\in S$, then $\frac{P(\X)}{T(\X)}\in {\mathfrak M}_z$ iff $P(z)=0$.
As $F$ is infinite it is well known that this implies that $P=0$. 
\vskip 1mm \vskip 0pt
8) Let ${\mathfrak M}\subset A$, be a maximal ideal; then $\Q\subset \frac{A}{{\mathfrak M}}$ and so is infinite.

\end{proof}

\begin{prop}\label{prop4.3}
Let $A$ be a field or a semi-local ring (i.e. the set 
of 
maximal ideals in $A$ is finite), then $A$ is a good ring.
\end{prop}

\begin{proof}
If $A$ is a field this is immediate from the definition. Now let us assume that 
the set of maximal ideals in $A$ is $\{\mathfrak M_1,\mathfrak 
M_2,...,\mathfrak M_r\}$ 
then 
the diagonal morphism $d:A\to \prod_{1\leq i\leq r}\frac{A}{\mathfrak M_i}$ 
induces 
an isomorphism of rings $\frac{A}{\mathfrak R}\simeq \prod_{1\leq i\leq 
r}\frac{A}{\mathfrak M_i}$. As $\frac{A}{\mathfrak M_i}$ is a good ring we 
deduce from 
example 
2.  in 4.1.
that 
its the same for $\frac{A}{\mathfrak R}$ and so for $A$ (Proposition 
\ref{prop4.1}).
\end{proof}

\begin{prop}\label{prop4.4}
Let $A$ be a principal ideal ring (PIR). We assume that for all $x\in A$ which 
is not 
a zero divisor, the quotient ring $\frac{A}{xA}$ is finite, then $A$ is a good 
ring.

\end{prop}
\begin{proof} 
By (\cite{ZS} Theorem 33 page 245), we know that a principal ideal ring $A$ is 
a 
finite product of rings $A=A_1\times A_2\times ...\times A_r$ where $A_i$ is a 
PID or a local ring whose the maximal ideal $\mathfrak M_i$ is generated by
  a nilpotent element.
  
If $A_i$ is a PID and $\mathfrak M_i=z_iA_i$ a maximal ideal, then 
$(1,..,1,z_i, 
1,...,1)\in A$ is not a zero divisor. As $\frac{A_i}{z_iA_i}\simeq 
\frac{A}{(1,..,1,z_i,1,...,1)A}$ is finite, then if $x_i\in A_i-\{0\}$, 
$\frac{A_i}{x_iA_i}$ is a finite ring and so $\frac{A_i}{x_iA_i}$ is finite for 
any $x_i\in A_i-\{0\}$. Now with Remark \ref{rem2.1} and Proposition 
\ref{prop2.1n}, $A_i$ is a good ring. 

If $A_i$ is a local ring, then $A_i$ is a good ring (cf. example 5 in 4.1 ). 
Finally $A$ is a finite product of good rings and so $A$ is a good ring (cf. 
example 2 in 4.1 ).

\end{proof}

\begin{prop}\label{prop4.5}
\begin{enumerate}[A.] Let $A$ be a ring.

 \item Let $\rho: \Z\to A$ be the natural homomorphism. We assume that any 
element in $A$ is integral over  $\rho (\Z)$, then $A$ is a good ring.
 \item Let $L$ be a subfield of an algebraic closure of a finite 
field, $L[T]$ the polynomial ring in the variable $T$ over $L$. We assume there 
is   $\rho: L[T]\to A$ an homomorphism with $A$ integral over 
$\rho(L[T])$, i.e. every element in $A$ is zero of a unitary polynomial with 
coefficient in $L[T]$, then $A$ is a good ring.
 
\end{enumerate}

\end{prop}

\begin{proof}

As $\Z$ (resp. $L[T]$) is a pictorsion ring (\cite{MB}, Theorem 2.3. p. 165), it follows by Corollary \ref{coro3.1} that $\Z$ (resp. $L[T]$) is a good ring.

Note that there is a direct algebraic proof of this result, due to its technical nature we refer to \newline \cite{FM}, Proposition 4.6.
\end{proof}

\subsection{Examples of rings $A$ which aren't good rings.}

\subsubsection{Good rings and localization}

In general a localisation of a good ring isn't a good ring as shown with the following example. 

\begin{prop} \label{Loc}
 Let $k$ be a field and $k[X,Y,Z]$, the polynomial ring. 
Let ${\mathfrak M}=(X,Y,Z)\subset k[X,Y,Z]$ be a maximal ideal and 
$A: =k[X,Y,Z]_{\mathfrak M}$, then $A$ is a good ring. 

Let $B:=A[\frac{1}{X+YZ}]$ be the localized ring $S^{-1}A$, with  $S:=\{(X+YZ)^n\ |\ n\in \N\}$,  then 
$B$ isn't a good ring.

\end{prop}

\begin{proof}
As $A$ is  a local ring it is  a good ring (Proposition \ref{prop4.3}).

Note that 
$A:=\{\frac{P}{R}\ |\ P\in k[X,Y,Z],\ R\in k[X,Y,Z],\ R(0,0,0)\neq 0\}$ and 
\newline 
$B=\{\frac{P}{(X+YZ)^nR}\ |\ P,R\in k[X,Y,Z], \ n\geq 0,\ R(0,0,0)\neq 0\}$. 

As $X+YZ$ is irreducible in the factorial ring $k[X,Y,Z]$, we have 
\begin{eqnarray}
 B^\times=(X+YZ)^\Z\times\{\frac{P}{Q}\ |\ P,Q\in k[X,Y,Z],\ P(0,0,0)\neq 0,Q(0,0,0)\neq 0\}. \label{1Loc}
\end{eqnarray}
Let $a:=X\in B$, then  $\frac{B}{aB}\simeq\{\frac{P}{(YZ)^nR}\ |\ P,R\in k[Y,Z], \ n\geq 0,\ R(0,0)\neq 0.\}$
It follows that 
\begin{eqnarray}
 (\frac{B}{aB})^\times = Y^\Z\times Z^\Z\times\{\frac{P}{R}\ |\ P,R\in k[Y,Z],\ R(0,0)\neq 0\} . \label{2Loc}
\end{eqnarray}
Let $\rho_a:B\to \frac{B}{aB}$, the natural epimorphism. It follows from \eqref{1Loc} that 
\begin{eqnarray}
 \rho_a(B^\times)=Y^\Z\times Z^\Z\times\{\frac{P}{R}\ |\ P,R\in k[Y,Z],\ R(0,0)\neq 0\} . \label{3Loc}
\end{eqnarray}
Now with \eqref{2Loc} and \eqref{3Loc} we have 
$\frac{\rho_a(B)^\times}{\rho_a(B^\times)}\simeq \frac{Y^\Z\times Z^\Z}{(YZ)^\Z}\simeq \Z.$

Then $\frac{\rho_a(B)^\times}{\rho_a(B^\times)}$ is not torsion and so $B$ isn't a good ring.
\end{proof}

\subsubsection{Good rings and transfert to polynomial rings}

A. {\sl  The case of polynomial rings $k[T]$ with coefficients in a 
field is answered in the following proposition.}

\begin{prop}\label{lem4.3}
Let $k$ be a field, $A:=k[T]$ the polynomial ring of the variable 
$T$ 
and coefficients in $k$. 

Then the following properties are equivalent.

\begin{enumerate}[i)]
 \item The ring $A$ is a good ring,
 \item the field $k$ is algebraic over a finite field,
 \item the group $k^\times$ is a torsion group.
\end{enumerate}
\end{prop}

\begin{proof}
1) {\sl We show that ii) implies iii).} 

We remark that $\F_p^{alg}=\cup_{n\geq 
1}\F_{p^n}$, it follows that $(\F_p^{alg})^\times$ and so $k^\times$ is a 
torsion group.
\vskip 1mm \vskip 0pt
\noindent 
2) {\sl We show that iii) implies ii).} 

Necessarily $k$ has characteristic $\neq 0$, 
otherwise $\Q\subset k$, this contradict the fact that $k^\times$ is a torsion 
group. One can assume that $\F_p\subset k$. Let $x\in k^\times$, there is 
$n_x\geq 1$ with $x^{n_x}=1$ and so $k$ is algebraic over $\F_p$. 
\vskip 1mm \vskip 0pt
\noindent 
3) {\sl We show that i) implies iii).} 

One can assume that  $k\neq\F_2$. Let 
$\theta\in k-\{0,1\}$, then  $1=\GCD (T-\theta,T(T-1))$ so there is $u,v\in 
k[T]$ with $u(T-\theta)+vT(T-1)=1$. Let $b=T-\theta$, $a:=T(T-1)$, as 
$k[T]$ is a good ring there is $m\geq 1$ which 
depends  on $\theta$ and $\lambda(T)\in k[T]$ with $b^m+\lambda(T)a=\epsilon 
\in 
k^\times =A^\times$ which a specialisation of $T$ by $0$ and $1$ gives
$(-\theta)^m=\epsilon,\ (1-\theta)^m=\epsilon$ and so 
$(\frac{1-\theta}{-\theta})^m=1$. Now we remark that $\theta \to 
\frac{1-\theta}{-\theta}$ is a permutation of $k-\{0,1\}$, it follows that 
$k^\times$ is a torsion group.
\vskip 1mm \vskip 0pt
\noindent 
4) {\sl We show that ii) implies i).} 

One can assume that $k\subset 
\F_p^{alg}$. 
Let $z\in k[T]$, we show that $(\frac{k[T]}{zk[T]})^\times $ is a torsion 
group. 

If $z=0$ as $k[T]^\times=k^\times$ then $(\frac{k[T]}{zk[T]})^\times =k^\times$ 
is a torsion group.

We now assume that $\deg z=d\geq 0$.  If $d=0$ then 
$(\frac{k[T]}{zk[T]})^\times=\{1\}$.

If $d\geq 1$, then $\frac{k[T]}{zk[T]}$ is a $k$-vector space $V$ of dimension 
$d\geq 1$. Let $f\in (\frac{k[T]}{zk[T]})^\times $, then the map \newline
$\mu_f:x\in V\to fx\in V $ belongs to $\Gl (V)$ and the map $f\in 
(\frac{k[T]}{zk[T]})^\times \to \mu_f\in \Gl (V)$ is an into homomorphism. 
After 
fixing a basis of the $k$-vector space $V$, the element $\mu_f$ and so $f$ can 
be identified with an element of $\Gl_d (k)$ and so for $n$ big enough with an 
element of $\Gl_d(\F_{p^n})$. It follows that $f$ is a torsion element. 

\end{proof}

\begin{rem}\label{rem4.1}
 It follows from Proposition \ref{lem4.3} that if $k$ is a field of 
characteristic $0$, then $k[T]$ isn't a good ring.
\end{rem}


\vspace{1\baselineskip}
\noindent 
B. {\sl The case of polynomial rings $A[T]$ with coefficients in an integral domain which is not a field
is answered in the following proposition.}

\begin{prop}\label{Tpolprop}
Let $A$ be an integral domain which is not a field, $B:=A[T]$ the polynomial ring of the variable 
$T$ and coefficients in $A$. Then $B$ isn't a good ring.
\end{prop}

\begin{proof}
As $A$ isn't a field, there is  $\pi \neq 0$ and $\pi\notin A^\times $.
Let $a:=1-\pi T$,  and $S:=\{\pi^k\ | \ k\in \N \}$, then $\card S$ is infinite. 

Let $\rho_a: A[T]\to \frac{A[T]}{aA[T]}$, the natural epimorphism. Then $\rho_a$ induces
an isomorphism 
$f:S^{-1}A\to \frac{A[T]}{aA[T]}$, which is defined by $f(\frac{\lambda}{\pi^n})=\rho_a(\lambda T^n)$. 

We have $S^{-1}A \supseteq \pi^\Z  A^\times=\pi^\Z \times A^\times $ and so 
$\rho_a(A[T])^\times \supseteq \rho_a(\pi)^\Z \times f(A^\times)=\rho_a(T)^\Z \times f(A^\times)$.

On the other hand, $A[T]^\times=A^\times$ and $ \rho_a(A[T]^\times )=f(A^\times)$.  It follows that 
$\rho_a(T)^\Z\subset \frac{\rho_a(A[T])^\times }{\rho_a(A[T]^\times )}.$ 

Now as $\pi^\Z 
=\rho_a(T)^\Z\simeq \Z$, we see that $B=A[T]$ isn't a good ring. 
 
\end{proof}
\vspace{1\baselineskip}
\noindent
C.  {The case of polynomial rings $A[T]$ with coefficients in a ring with zero Jacobson radical 
is answered in the following proposition.}

\begin{prop}\label{TpolRJprop} Let $A$ be a ring with Jacobson radical nul. Then
\begin{enumerate}
 \item  If $\dim A\geq 1$, then $A[T]$ isn't a good ring.
 \item  If $\dim A=0$ and if there is a maximal ideal ${\mathfrak M}\subset A$ such that 
 $\frac{A}{\mathfrak M}$ isn't  contained in the algebraic closure of a finite field,   
 then $A[T]$ isn't a good ring.
\end{enumerate}
\end{prop}

\begin{proof}
1) There is a prime ideal $\mathfrak p\subset A$ which isn't maximal, i.e. $\frac{A}{\mathfrak p}$ is an
integral domain and not a field. Then by \ref{Tpolprop} we know that $\frac{A}{\mathfrak p}[T]$ isn't
a good ring and then from stability by quotients  (4.1.1. part 3.), the ring $A$ isn't a good ring.

\noindent 2) 
If $\frac{A}{\mathfrak M}$ isn't  contained in the algebraic closure of a finite field, by 
proposition \ref{lem4.3} that $\frac{A}{\mathfrak M}[T]$ isn't a good ring and from stability by quotients  
(4.1.1. part 3.), the ring $A[T]$ isn't a good ring.
\end{proof}

\begin{rem}
The preceding proposition doesn't give an answer when  $\dim A=0$ and for any maximal ideal ${\mathfrak M}$ 
 of $A$, the residue field $\frac{A}{\mathfrak M}$ is contained in the algebraic closure of a finite field.
 In fact one can give many examples of such a ring $A$ with $A[T]$ a good ring. 
 
 Namely, let $A:=K_1\times K_2\times ...\times K_s$, where $K_i$ is a field which is algebraic over a finite field for $1\leq i\leq s$. Then $A[T]\simeq\
 K_1[T]\times K_2[T]\times ...\times K_s[T]$ and it follows from Proposition \ref{lem4.3}
 that $K_i[T]$ is a good ring  and by 4.1.1. part 2. that $A[T]$ is a good ring.

Another example is the ring $A:=\F_p^\N$.

\end{rem}

\subsection{Good rings  and pictorsion}\label{ssec4.3} See Definition\ref{def.Good} and Definition \ref{def.Pictorsion}

In the sequel we discuss the relation between good rings and torsion in Picard groups.
\vskip 1mm \vskip 0pt
\noindent
A. {\sl  Rings which aren't of pictorsion.}
 
We saw in Corollary
\ref{coro3.1} that a pictorsion ring is a good ring and 
rings which aren't good rings aren't of pictorsion. 
\begin{rem}
For example the ring $A=\Q[y]$ isn't a good ring (Proposition \ref{lem4.3}) and so isn'a pictorsion ring. 
In this case $\Pic(A)$ is trivial and so there is a finite 
extension $B$ over $A$ such that the group $\Pic (B)$ is not a torsion group. One can give an explicit example 
of such an extension $B/A$ (compare with \cite{GLL}, Example 8.15 p. 1263). 

Namely, Let $A:=\Q[y]$  be the ring of polynomial in the variable $y$ with coefficients in 
$\Q$. Then $\Pic (A)$ is trivial.

Let $B:=\frac{\Q[y,X]}{(X^3-X-y^2-y)\Q[y,X]}=\Q[y,x]$ where $x$ is the $X$ 
image.  The ring $B$ is finite over $A$. We show that $\Pic(B)$ is not a 
torsion 
group and so $A$ is not of pictorsion. 

Let $\Q[x,y,z]:=\frac{\Q[X,Y,Z]}{(X^3-XZ^2-Y^2Z-YZ^2)\Q[X,Y,Z]}$, where $x$ 
(resp. $y$, $z$) is the $X$ (resp. $Y$,$Z$) image and the grading is induced 
by that of $\Q[X,Y,Z]$.

Let $S:=\Proj (\Q[x,y,z])$, $\infty:=(0:1:0)$. We have $S-\{\infty\}=D_+(z)$ 
and 
${\cal O}_S(S-\{\infty\})=B$. 

Let $p:=(0:0:1)$, the divisor $(p-\infty)$ has infinite order (\cite{H}, Example 
4.3.8. p. 335). It follows that the ideal $\mathfrak M_p=Bx+By$ induces an 
element 
of infinite order in $\Pic (B)$. 
\end{rem}

\noindent
B. {\sl  Good rings with a finite Picard group.}

i) Let $A$ be the ring of integers of a number fields or the ring of regular functions 
on a smooth affine curve over a finite field. Such a ring is a good ring (see 4.1.2. example 3) and 
$\Pic (A)$ is finite. Moreover for the same reason such a ring is of pictorsion.  

ii) Let us consider the good rings in examples 5 in paragraph 4.1.3. In contrary  to the examples in i),
the residue
fields aren't finite and nevertheless one can give examples with a finite Picard group. 

Namely, let $X$ be an algebraic non singular curve over $\Q$; let us assume that 
$X(\Q)\neq \emptyset$. Let $a\in X(\Q)$ and $Y:=X(\Q)-\{a\}$. 
Let  $A:=\{f\in 
K(X)\ |\ v_y(f)\geq 0,\ \forall y\in Y\}$,  where $K(X)$ is the field of rational 
functions on $X$ and $v_x$ is the valuation at $x$. Let us assume that $\Jac (X)(\Q)$ is   finite 
(see \cite{GG}, Theorem 3, p. 822 for a genus $2$ example and \cite{BS} for the existence of infinitely many 
$\Q$-elliptic curves). 

Let $U:=X-\{a\}$, then  $B:={\cal O}_X(U)$ is a Dedekind domain. Moreover  we have the following exact 
sequence $$\{0\}\to \Z a\to \Pic(X)\to \Pic(U)\to \{0\}$$ and
as $\Pic(X)=\Pic^0(X)\oplus \Z a$ we get that $\Pic(U)\simeq \Pic^0(X)$. Now as $\Pic^0(X)\subset \Jac (X)(\Q)$
it follows that $\Pic(U)$ is finite and then by \cite{G}, Corollary 1, p. 114, there is $S\subset B$, a multiplicative 
set such that $A=S^{-1}B$. In particular $\Pic (A)$ is finite. Note that we don't know if the ring $A$ is a pictorsion
ring.
\vskip 1mm \vskip 0pt
\noindent
C. {\sl  Good rings with a non torsion Picard group and so which aren't of pictorsion.}

Let us recall (Corollary \ref{coro3.1}) that a pictorsion ring is a good ring. 
\vskip 1mm \vskip 0pt
There is still the question to find a good ring $A$  such 
that the Picard group of $A$ is not a torsion group. This is what we intend to 
do now.

\begin{prop}\label{prop4.6}
Let $A$ be a ring, $a\in A$,
$B_a:=\frac{A[T]}{T(T-a)A[T]}=A\oplus At$ where $t$ is the $T$ image. Let 
$\rho_a:A\to \frac{A}{aA}$ the natural epimorphism. Then we have the following 
exact sequence 
$$\{1\}\to \frac {\rho_a(A)^\times}{\rho_a(A^\times)}\to \Pic (B_a)\to \Pic 
(A)^2\to \Pic(\frac{A}{aA}).$$
In particular, if $\frac {\rho_a(A)^\times}{\rho_a(A^\times)}$  and $\Pic (A)$  
are torsion groups then $ \Pic (B_a)$ is a torsion group.
\end{prop}

\begin{proof}
1) {\sl Let $I:=\{b\in A\ |\ ba=0\}$ and $C_a:=\frac{B_a}{It}$. We show that 
$\Pic (B_a)\simeq\Pic (C_a)$.}

Let $X:=\Spec A$, $Y:=\Spec B_a$, $Z:=\Spec C_a$. The canonical morphisms $A\to 
B_a$ and $A\to B_a\to C_a$ induce continuous maps $f:Y\to X$ and $g:Z\to X$.

We have the exact sequence of sheaves $\{0\}\to {\cal I}t\to {\cal F}\to{\cal 
G}\to \{0\}$ where $\cal F$ is the sheaf defined over $X$ by the ideal $I$, and 
for all open $U\subset X$ we have ${\cal F}(U):={\cal O}(f^{-1}(U))$, ${\cal 
G}(U):={\cal O}(g^{-1}(U))$.

It follows the following exact sequence of sheaves
\begin{eqnarray}
 \{1\}\to 1+{\cal I}t\to {\cal F}^\times \to {\cal G}^\times\to \{0\} 
\label{(2g)}
\end{eqnarray}
where ${\cal F}^\times(U):={\cal O}(f^{-1}(U))^\times $, ${\cal 
G}^\times(U):={\cal O}(g^{-1}(U))^\times $.

Then~\eqref{(2g)} gives the long exact sequence of sheaves
\begin{eqnarray}
\{0\}\to 1+It\to B_a^\times\to C_a^\times\to \to H^1(X,{\cal I})\to \Pic (B_a)\to 
\Pic (C_a)\to H^2(X,{\cal I}).  \label{(3g)}
\end{eqnarray}
As $ 1+{\cal I}t\simeq  {\cal I}$, we have $ H^1(X,{\cal I})= H^2(X,{\cal 
I})=\{0\}$ (\cite{L}, Theorem 2.18, p. 186).

It follows from~\eqref{(3g)} that $\Pic (B_a)\simeq\Pic (C_a)$.
\vskip 1mm \vskip 0pt
\noindent
2) {\sl We show the exact sequence 
 $$ \{1\}\to \frac {\rho_a(A)^\times}{\rho_a(A^\times)}\to \Pic (C_a)\to \Pic 
(A)^2\to \Pic (\frac{A}{aA}), $$
where the map $\Pic (C_a)\to \Pic (A)^2$ is induced by 

$u:B_a=A\oplus At\to 
A\times A$ with $u(\alpha+\beta t)=(\alpha,\alpha+\beta t).$}

We have $\Ker u=\{\beta t\ |\ \beta a=0\}=It$. So $u$ induces an into 
homomorphism $v:C_a\to A\times A$. So $v$ induces a morphism ${\cal O}_Z \to 
{\cal O}_X\times {\cal O}_X$ and so a morphism ${\cal O}_Z^\times \to {\cal 
O}_X^\times \times{\cal 
O}_X^\times.$

Let now $W:=\Spec(\frac{A}{aA})$; then the homomorphism $(\ell ,m)\to 
\rho_a(\ell)\rho_a(m^{-1})$ induces a morphism ${\cal O}_X^\times \times{\cal 
O}_X^\times\to {\cal O}_W^\times.$ We have the exact sequence

$$\{1\}\to {\cal O}_Z^\times\to {\cal O}_X^\times \times{\cal O}_X^\times\to 
{\cal 
O}_W^\times \to \{1\}$$ from which we get the long exact sequence 
\begin{eqnarray}
 \{1\}\to C_a^\times \to A^\times\times A^\times\to (\frac{A}{aA})^\times\to 
\Pic (C_a)\to \Pic (A)^2\to \Pic (\frac{A}{aA}).  \label{(5g)}
\end{eqnarray}
It induces an into homomorphism $\frac {\rho_a(A)^\times}{\rho_a(A^\times)}\to 
\Pic C_a.$

It follows from~\eqref{(5g)} the exact sequence 
$$\{1\}\to \frac {\rho_a(A)^\times}{\rho_a(A^\times)}\to \Pic C_a\to \Pic 
(A)^2\to \Pic (\frac{A}{aA}).$$
\vskip 1mm \vskip 0pt
\noindent
3) It follows from 1) and 2) the exact sequence
\begin{eqnarray}
 \{1\}\to \frac {\rho_a(A)^\times}{\rho_a(A^\times)}\to \Pic B_a\to \Pic 
(A)^2\to \Pic (\frac{A}{aA}). \label{(7g)}
\end{eqnarray}

If the groups $\frac {\rho_a(A)^\times}{\rho_a(A^\times)}$ and $\Pic (A)$ are 
torsion groups then~\eqref{(7g)} implies that $\Pic B_a$ is also a torsion 
group. 
\end{proof}

\begin{prop}\label{prop4.7}
(\cite{G}, Corollary 2) Let $\Z[T]$ the ring of polynomials of the variable 
$T$ with coefficients in $\Z$. Then there is a Dedekind domain $A$ with 
$\Z[T]\subset A\subset \Q[T]$ such that for each maximal ideal $\mathfrak 
M\subset 
A$, the 
field $\frac{A}{\mathfrak M}$ is finite and the group $A^\times $ is of finite 
type. 
Moreover the class group of $A$ is not a torsion group, i.e. $\Pic (A)$ is not 
a 
torsion group.
\end{prop}

\begin{prop}\label{prop4.8}
Let $A$ be the ring defined in \ref{prop4.7}. Then $A$ is a good ring.

Let $(a,b)\in A^2$ a primitive point, then the Picard group of ${\cal 
O}_{S(a,b)}(S(a,b))$
is not a torsion group and $A$ is not a pictorsion ring.
\end{prop}

\begin{proof}
1) As $A$ is a Dedekind domain and that for each maximal ideal $\mathfrak M\subset A$, 
the 
ring 
$\frac{A}{\mathfrak M}$ is finite, it follows that $\rho_a(A)$ is finite for 
all 
$a\neq 0$ (Proposition \ref{prop4.1}), and so $\frac 
{\rho_a(A)^\times}{\rho_a(A^\times)}$ 
is finite for all $a\neq 0$ and  trivially $\frac 
{\rho_a(A)^\times}{\rho_a(A^\times)}=\{1\}$ if $a=0$. In particular $A$ is a 
good ring.
\vskip 1mm \vskip 0pt
\noindent
2) We have ${\cal 
O}_{S(a,b)}(S(a,b))=B_a:=\frac{A[T]}{T(T-a)A[T]}=A\oplus At$ (Proposition 
\ref{prop3.1}, part 
1.).

It follows from Proposition \ref{prop4.6} the following exact sequence
\begin{eqnarray}
 \{1\}\to \frac {\rho_a(A)^\times}{\rho_a(A^\times)}\to \Pic B_a\to \Pic 
(A)^2\to \Pic (\frac{A}{aA}). \label{(1h)}
\end{eqnarray}

2.1) Let us assume that $a\neq0$. It  follows that $\frac{A}{aA}$ is finite and 
so that $\Pic (\frac{A}{aA})$ is finite. Moreover as $\Pic (A)$ is not a 
torsion group it follows from~\eqref{(1h)} that $\Pic (B_a) $ is not a torsion 
group, i.e. the Picard group of ${\cal O}_{S(a,b)}(S(a,b))$
is not a torsion group.

2.2) Let us assume that $a=0$. This means that 
$${\cal O}_{S(0,b)}(S(0,b))=B_0:=\frac{A[T]}{T^2A[T]}=A\oplus At\ {\rm with}\ 
t^2=0.$$

We show that $\Pic (A)\simeq \Pic (B_0)$. 

Let $X:=\Spec A$, $Y:=\Spec B_0$. From the exact sequence 
$$ \{0\}\to tA\to B_0\to \frac{B_0}{tB_0}\simeq A\to \{0\},$$
we get the following exact sequence of sheaves 

\vskip 2mm \vskip 0pt
${1}\to 1+tA\to B_0^\times\to A^\times \to H^1(X,1+t{\cal O}_X)\to 
\to \Pic 
(B_0)\to\Pic (A) \to H^2(X,1+t{\cal O}_X).$

\vskip 2mm \vskip 0pt
As $t{\cal O}_X$ is coherent, $H^1(X,1+t{\cal 
O}_X)\simeq H^1(X,t{\cal O}_X)=\{0\}$ and 
$H^2(X,1+t{\cal O}_X)\simeq 
H^2(X,t{\cal O}_X)=\{0\}$. 

It follows that $\Pic (A)\simeq \Pic (B_0)$ and so 
that the Picard group of ${\cal O}_S(S(0,b))$
is not a torsion group.
\end{proof}

\bibliographystyle{alpha}

\begin{thebibliography}{Aue00}

\bibitem[B]{B}
Bass Hyman
\newblock {\em $K$-theory and stable algebra}.
\newblock {\em {\rm Publi. Math. IHES}}, 22 (1964) 5-60

\bibitem[Be.E]{BE}
Berman John, Erman Daniel
\newblock {\em Interpolation over $\Z$ and torsion in class groups}.
\newblock {\em {\rm arXiv:1905.12097v2}}, (2020)

\bibitem[Br.E]{BrE}
Bruce Juliette, Erman Daniel
\newblock {\em A probalistic approach to systems of parameters and Noether normalization}.
\newblock {\em {\rm Algebra Number Theory }}, 13 (2019), no. 9, 2081-2102.

\bibitem[B.S]{BS}
Bhargava Manjul, Shankar Arul 
\newblock {\em Ternary cubic forms having bounded invariants, and the
existence of a positive proportion of elliptic curves having rank
0}.
\newblock {\em {\rm Ann. of Math. (2)}},  Vol. 181 (2015), no. 2 (2015) 587–621.

\bibitem[C.MB.P.T]{CMPT}
Chinburg Ted, Moret-Bailly Laurent, Pappas Georgios, Taylor Martin
\newblock {\em Finite morphisms to projective space and capacity theory}.
\newblock {\em {\rm J. reine angew. Math.}}, 727 (2017), 69-84

\bibitem[F.M]{FM}
Fresnel Jean, Matignon Michel
\newblock {\em Good rings and homogeneous polynomials}.
\newblock {\em {\rm arXiv:1907.05655v2}}, (2019)


\bibitem[G.L.L]{GLL}
Gabber Offer, Liu Qing, Lorenzini Dino 
\newblock {\em Hypersurfaces in projective schemes and moving lemma}.
\newblock {\em {\rm Duke Mathematical Journal}} Vol. 164, No.7(2015) 1187-1270 


\bibitem[G]{G}
Goldman Oscar
\newblock {\em On special a class of Dedekind domain}.
\newblock {\em {\rm Topology}}
vol. 3, suppl. 1 p. 113-118 Pergamon Press (1964)

\bibitem[G.G]{GG}
Gordon Daniel, Grant David
\newblock {\em Computing the Mordell-Weil rank of Jacobians of curves of genus two,}
\newblock {\em {\rm Trans. Amer. Math. Soc.}} 337 (1993), 807-824.

\bibitem[G.D]{GD}
Grothendieck Alexandre,  Dieudonn\'e Jean
\newblock {\em E.G.A. chap. II, \'Etude globale \'el\'ementaire de
classes de morphismes}.
\newblock {\em {\rm Publi. Math. IHES}}, 8 (1961) 

\bibitem[H]{H}
Hartshorne Robin 
\newblock {\em Algebraic Geometry}.
\newblock {\em {\rm Graduate texts in Mathematics}} Springer-Verlag (1977)

\bibitem[K.L.W]{KLW}
Khurana Dinesh, Lam T.Y., Wang Zhou
\newblock {\em Rings of square stable range one}.
\newblock {\em {\rm Journal of Algebra}} 338 (2011) 122-141 

\bibitem[L]{L}
Liu Qing
\newblock {\em Algebraic Geometry and Arithmetic curves}.
\newblock {\em {\rm Oxford  Graduate texts in Mathematics}} (2002)

\bibitem[MB]{MB}
Moret-Bailly Laurent
\newblock {\em Groupes de Picard et problèmes de Skolem. I} 
\newblock {\em{\rm Ann. E. N. S.}} $4^e$  série,tome 22, no. 2 (1989),161-179. 

\bibitem[Z.S]{ZS}
Zariski Oscar, Samuel Pierre 
\newblock {\em Commutative algebra}.
\newblock {\em {\rm Graduate texts in Mathematics }} Springer-Verlag  (1975)


\end{thebibliography}

Jean Fresnel, 

Univ. Bordeaux, CNRS, Bordeaux INP, IMB, UMR 5251, F-33400, 
Talence, France

\vskip 2mm \vskip 0pt
\noindent 
Michel Matignon,

Univ. Bordeaux, CNRS, Bordeaux INP, IMB, UMR 5251, F-33400,
Talence, France

\vskip 2mm \vskip 0pt
E-mail address: Jean.Fresnel$@$math.u-bordeaux.fr  

E-mail address: Michel.Matignon$@$math.u-bordeaux.fr

\end{document}